\documentclass[a4paper]{amsart}
\usepackage{amssymb}

\usepackage{yfonts}

\usepackage{a4wide,enumerate,graphicx,fancyhdr}
\usepackage[normalem]{ulem} 

\usepackage{theoremref} 

\usepackage{mathrsfs}
\usepackage{tabulary}
\usepackage[all,cmtip]{xy}

\usepackage{enumitem}

\usepackage{amsrefs} 
\usepackage{tikz} 

\newtheorem{thm}{Theorem}[section]

\newtheorem{lem}[thm]{Lemma}
\newtheorem{prop}[thm]{Proposition}
 
\newtheorem{cor}[thm]{Corollary} 
 
\newtheorem{conj}[thm]{Conjecture}

\theoremstyle{definition}
\newtheorem{rem}[thm]{Remark}
 
\newtheorem{defn}[thm]{Definition}
\newtheorem{constr}[thm]{Construction} 
 
\newtheorem{ex}[thm]{Example} 
 
\newtheorem{conv}[thm]{Convention} 
\newtheorem{cond}[thm]{Condition} 
\newtheorem{setting}[thm]{Setting}
\newtheorem{question}[thm]{Question}

\newtheorem*{thm*}{Theorem}

\newcommand{\id}{{\rm id}}

\newcommand{\scal}{scal}
\newcommand{\N}{\mathbb{N}}
\newcommand{\C}{\mathbb{C}}
\newcommand{\F}{\mathbb{ F}}

\newcommand{\Z}{\mathbb{ Z}}

\newcommand{\lra}{\longrightarrow}

\newcommand{\HH}{\operatorname{H}}

\newcommand{\RH}{\operatorname{RH}} 
 
\newcommand{\BP}{\operatorname{BP}}

\newcommand{\MU}{\operatorname{MU}}

\newcommand{\Tor}{\operatorname{Tor}}
\newcommand{\SO}{\operatorname{SO}}

\newcommand{\R}{\mathbb{ R}}
 
\newcommand{\im}{{\rm im\,}}
 
\newcommand{\pr}{\operatorname{pr}} 

\newcommand{\RP}{\mathbb{ R}\mathbb{ P}}

\newcommand{\ttimes}{\tilde{\times}}

\newcommand{\CP}{\mathbb{ C}{\rm P}}

\newcommand{\be}{\begin{eqnarray*}}
\newcommand{\ee}{\end{eqnarray*}}
\newcommand{\ba}{\begin{align*}}
\newcommand{\ea}{\end{align*}}

\newcommand{\sing}{\mathrm{Sing}}

\newcommand{\rot}[1]{\textcolor{black}{#1}}

\BibSpec{collection.article}{%
    +{}  {\PrintAuthors}                {author}
    +{,} { \textit}                     {title}
    +{.} { }                            {part}
    +{:} { \textit}                     {subtitle}
    +{,} { \PrintContributions}         {contribution}
    +{,} { \PrintConference}            {conference}
    +{}  {\PrintBook}                   {book}
    +{,} { }                            {booktitle}
    +{,} { }                            {publisher}
    +{,} { \PrintDateB}                 {date}
    +{,} { pp.~}                        {pages}
    +{,} { }                            {status}
    +{,} { \PrintDOI}                   {doi}
    +{,} { available at \eprint}        {eprint}
    +{}  { \parenthesize}               {language}
    +{}  { \PrintTranslation}           {translation}
    +{;} { \PrintReprint}               {reprint}
    +{.} { }                            {note}
    +{.} {}                             {transition}
    +{}  {\SentenceSpace \PrintReviews} {review}
}

\usepackage{hyperref}

\begin{document}

\title[Positive scalar curvature and odd order abelian fundamental groups]{Positive scalar curvature on manifolds with odd order abelian fundamental groups} 
\author{Bernhard Hanke}
\address{Institut f\"ur Mathematik, Universit\"at Augsburg, D-86135 Augsburg, Germany}
\date{\today}
\email{\href{mailto:hanke@math.uni-augsburg.de}{hanke@math.uni-augsburg.de}}

\begin{abstract} We introduce Riemannian metrics of positive scalar curvature on manifolds with Baas-Sullivan singularities, prove a corresponding homology invariance principle and discuss admissible products. 

Using this theory we construct positive scalar curvature metrics on closed smooth manifolds of dimension at least five which have odd order abelian fundamental groups, are nonspin and  atoral. 
This   solves the Gromov-Lawson-Rosenberg conjecture for a new class of manifolds with finite fundamental groups.

\end{abstract} 


\maketitle
\tableofcontents

\section{Summary} 

We will show the following existence result for positive scalar curvature metrics. 

\begin{thm} \label{main} Let $M$ be a closed  connected smooth manifold of dimension at least $5$  with odd order abelian fundamental group. 
Assume that $M$ is nonspin and $p$-atoral for all primes $p$ dividing the order of $\pi_1(M)$. 
Then $M$ admits a Riemannian metric of positive scalar curvature.
\end{thm} 

For the notion of $p$-atorality see Definition \ref{atoral}. 
For example, a  closed connected oriented manifold $M$ whose fundamental group is generated by fewer than $\dim(M)$ elements is $p$-atoral for all odd primes $p$, see Remark  \ref{ex:atoral} \ref{estrank}.

Theorem \ref{main}  contributes to the Gromov-Lawson-Rosenberg conjecture concerning  the existence of positive scalar curvature metrics on closed smooth manifolds, see  Rosenberg \cite{Ros_Surv}*{Conjecture 1.22}. 
It solves Problem 5.11 of Botvinnik and Rosenberg \cite{BR1} for odd $p$. 

For finite fundamental groups of odd order the Gromov-Lawson-Rosenberg conjecture can be formulated in the following concise way, see Rosenberg \cite{Ros}*{Conjecture 1.2}. 

\begin{conj} \label{ros} Let $M$ be a closed connected smooth manifold with finite  fundamental group of odd order. 
If the universal cover of  $M$ admits a positive scalar curvature metric, then $M$ admits a positive scalar curvature metric. 
\end{conj} 

Connected manifolds with odd order fundamental groups are orientable, and they are spin if and only if their universal covers are spin. 
Furthermore simply connected closed nonspin manifolds  of dimension at least $5$ admit positive scalar curvature metrics by Gromov and Lawson \cite{GL}*{Corollary C}. 
Hence, if Conjecture \ref{ros} holds, then each closed connected nonspin smooth manifold of dimension at least $5$ with fundamental group of odd order admits a positive scalar curvature metric, thus strengthening Theorem \ref{main}. 

Conjecture \ref{ros} holds in dimensions $1$ and $2$, and it holds in dimension $3$ by the geometrization theorem. 
In dimension $4$ it is false -- see Hanke, Kotschick and Wehrheim \cite{HKW} -- hence this case must be excluded. 
In  dimensions larger than or equal to $5$ it  holds for $p$-atoral manifolds whose fundamental groups are elementary abelian $p$-groups, where $p$ is an odd prime. 
This result is due to Botvinnik and Rosenberg \cites{BR1, BR2} and the author \cite{Ha15}, who discovered and corrected a gap in the original argument in \cites{BR1, BR2}.  

By  Kwasik and Schultz \cite{Kwasik}*{Theorem 1.8}, which can be generalized  to the nonspin case,  Conjecture \ref{ros} holds for manifolds of dimension larger than or equal to  $5$ whose fundamental  groups have periodic cohomology.

Conjecture \ref{ros} is false without assuming that $\pi_1(M)$ is of odd order, see the remarks after \cite{Ros}*{Theorem 1.3}.
Both this fact and the failure of the conjecture in dimension $4$ illustrate that the metric obtained from $\pi_1(M)$-averaging a positive scalar curvature metric on the universal cover of $M$  is in general not of positive scalar curvature.
Conjecture \ref{ros} remains open in general in dimensions larger than $4$.

\begin{defn} \label{atoral} 
Let $X$ be a topological space and let $p$ be a prime. 
A homology class $h \in \HH_d(X;\Z)$ is called {\em $p$-toral}, if there exist $\ell \in \N$, $\ell \geq 1$, and classes $c_1, \ldots, c_d \in \HH^1(X ;  \Z/ p^\ell )$ such that 
\[
     \big( c_1 \cup \cdots \cup c_d \big)(  h  )  \neq 0 \in \Z/p^\ell. 
\]  
Otherwise $h$ is called {\em $p$-atoral}.  

A closed oriented manifold $M$ of dimension $d$ is called {\em $p$-atoral} or {\em $p$-toral}, respectively, if the fundamental class $[M] \in \HH_d(M ; \Z)$  has the corresponding property.
\end{defn} 

\begin{rem}   \begin{enumerate}[label={(\roman*)}] 
\item  The $d$-torus $T^d = (S^1 \times \cdots \times S^1)^d$ for $d \geq 1$  is $p$-toral for all $p$, and so are all closed manifolds which are oriented bordant, over the classifying space $B (\Z/p)^d$, to the canonical map $T^d = B \Z^d \to B (\Z/p)^d$. 
\item  The $p$-atoral homology classes form a subgroup of $\HH_d(X;\Z)$. 
\item \label{classifying} A closed connected oriented manifold $M^d$ is $p$-toral if and only if $\phi_*([M]) \in \HH_d(B \pi_1(M);\Z)$ is $p$-toral, where $\phi: M \to B\pi_1(M)$ is the classifying map of the universal cover of $M$. 
This uses the fact that $\phi^* : \HH^1( B \pi_1(M) ; \Z / p^{\ell}) \to \HH^1(M ; \Z / p^{\ell})$ is an isomorphism for all  $\ell \geq 1$. 
\item \label{Sylow} Let $M^d$ be closed connected oriented with finite abelian fundamental group $\pi_1(M)$. 
Let $\psi : \overline{M} \to M$ be a connected cover corresponding to a  Sylow $p$-subgroup of $\pi_1(M)$. 
Then $M$ is $p$-toral, if and only if $\overline{M}$ is $p$-toral.
This follows from the relation
\[
     \big( \psi^*(c_1) \cup \cdots \cup \psi^*( c_d)  \big)(  [\overline{M}]  ) = \deg(\psi) \cdot  \big( c_1 \cup \cdots \cup c_d  \big)(  [M]  )
\]  
and from the fact that $\psi^* : \HH^1(M ; \Z/p^{\ell}) \to \HH^1(\overline{M} ; \Z/ p^{\ell})$ is an isomorphism for all $\ell \geq 1$ by our assumption on $\pi_1(M)$.
\item  \label{estrank}  Let $M^d$ be closed connected oriented and let $p$ be an  odd prime. 
Furthermore assume  that $\pi_1(M)$ is generated by fewer than $d$ elements. 
This implies that the abelianization of $\pi_1(M)$ is a product of fewer than $d$ cyclic groups. 
Then  $M$ is $p$-atoral since for $\ell \geq 1$ the cohomology group $\HH^1(M ; \Z/p^\ell )$ is generated by fewer than $d$ elements and  each element in $\HH^1(M ; \Z/p^{\ell})$ has square zero for odd $p$.  
\item In contrast the orientable real projective space $\RP^{2m-1}$ is $2$-toral for all $m \geq 1$. 
\item  One may speculate that $p$-toral manifolds for odd $p$ do not admit positive scalar curvature metrics. 
This would  yield  counterexamples to Conjecture \ref{ros}.  
\end{enumerate}
 \label{ex:atoral}
\end{rem}

In the spirit of other existence results for positive scalar curvature metrics on high dimensional manifolds the proof of Theorem \ref{main} is based on the propagation of positive scalar curvature metrics along surgeries of codimension at least three; see Gromov and Lawson \cite{GL} and Schoen and Yau \cite{SY}. 
In the  paper at hand this  technique is combined with the realization of singular homology classes by manifolds with Baas-Sullivan singularities \cite{baas1}. 
To this end we introduce and discuss the concept of {\em positive scalar curvature metrics on manifolds with Baas-Sullivan singularities} in Sections \ref{Baas_Sul} and \ref{admissible_prod}, which includes the discussion  of admissible products. 
In some particular cases positive scalar curvature metrics on simply connected manifolds with Baas-Sullivan singularities were studied by Botvinnik \cite{Bot01}. 

The main steps of the proof of Theorem \ref{main} are as follows. 
Let $\Omega^{\SO}_*$ denote the oriented bordism ring and fix a family $\mathscr{Q} = (Q_{4i})_{i \geq 1} $ of closed oriented manifolds of dimension $4i$ whose bordism classes form a set of polynomial generators of $\Omega^{\SO}_* / {\rm torsion}$, and each of which  is equipped with a metric of positive scalar curvature. 
Such families exist by the results in \cite{GL}. 
By \cite{baas1}, after inverting $2$ oriented bordism with singularities in $\mathscr{Q}$ is naturally isomorphic to singular homology; see Section \ref{background_BS}.  

Given a topological space $X$ we will define a subgroup  $\HH_*^{ \mathscr{Q},+}(X ;\Z) \subset \HH_*(X; \Z)$, called the  {\em positive homology of $X$ with respect to $\mathscr{Q}$}, see Definition \ref{poshom}. 
 Elements in this group are represented by maps from Baas–Sullivan manifolds admitting positive scalar curvature metrics to $X$. 
In particular positive homology classes need not be representable by smooth manifolds.
An important ingredient  for the  proof of Theorem \ref{main} is the following homology invariance principle, which we show at the end of Section \ref{Baas_Sul}.

\begin{thm} \label{maintwo} Let $M$ be a closed connected oriented smooth manifold of dimension $d \geq 5$ with odd order fundamental group and which is nonspin. 
Let $\phi : M \to B \pi_1(M)$ be the classifying map. 
Then $M$ admits a metric of positive scalar curvature if and only if $\phi_*([M]) \in \HH^{ \mathscr{Q} ,+}_d(B \pi_1(M); \Z)$.
\end{thm} 

It hence remains to show that under the conditions of Theorem \ref{main} we have $\phi_*([M]) \in \HH^{ \mathscr{Q} ,+}_d(B \pi_1(M); \Z)$ for some orientation $[M]$ of $M$. 
For this goal we first study  the positive homology $\HH^{\mathscr{Q}, +}_*(B \Gamma; \Z)$ for finite abelian $p$-groups $\Gamma$.  

In this case the homology of $B \Gamma$ can inductively be computed by an exact K\"unneth sequence (with $\alpha \geq 1$) 
\[
    0 \to \HH_*(B\Gamma)  \otimes \HH_*(B \Z/p^{\alpha}) \stackrel{\times}{\longrightarrow} \HH_*(B  \Gamma \times B\Z/p^{\alpha} )
   \longrightarrow \ \Tor ( \HH_*(B \Gamma), \HH_*(B \Z/p^{\alpha} ))_{*-1}  \to 0 \, . 
 \]
The cross product can be realized by admissible products of manifolds with Baas-Sullivan singularities,  and the same is true for the torsion product, which is related to a homological Toda bracket. 
The construction of admissible products and Toda brackets for Baas-Sullivan manifolds with positive scalar curvature is non-trivial and will be developed  in Sections \ref{admissible_prod} and \ref{Toda} of our paper. 

By a variant of the well known  ``shrinking one factor'' argument (see Proposition \ref{posproduct}) the cross product of two homology classes is positive, if one of the factors is positive.    
However we can in general show positivity  of Toda brackets only if {\em both}  of the factors are positive; compare Corollary \ref{summary}. 
This does not cover Toda brackets involving homology classes of degree one (represented by circles), even though  these Toda brackets are $p$-atoral.
Although we can show the positivity of many Toda brackets involving degree one classes by a systematic use of group homomorphisms in Proposition \ref{calculate}, there are  some Toda brackets  whose positivity remains obscure; see Question \ref{openprob}.  

In order to bypass this  issue we use the fact that the homology class $\phi_*([M]) \in \HH_d(B \pi_1(M); \Z)$ is of a restricted type, since $M$ is assumed to be a smooth manifold. 
This fact is explored in the following result.

\begin{thm} \label{representable} Let  $p$ be an odd prime and let $\Gamma$ be a finite abelian $p$-group.  
Then all $p$-atoral classes  in the image of $\Omega^{SO}_*(B \Gamma) \to \HH_*(B \Gamma; \Z )$ are positive. 
\end{thm} 

The proof of Theorem \ref{representable} will be provided in Section \ref{CF}. 
As a preparation we investigate the (ordinary) homology of abelian $p$-groups $\Gamma$ in Section \ref{hom_p_1}. 
For $\Gamma =  (\Z/p^{\alpha})^n$ and $p$-atoral homology classes not divisible by $p$ the proof of Theorem \ref{representable} is especially difficult and relies on the fact that these homology classes  can be represented by generalized products of $\Z/p^{\alpha}$-lens spaces modulo elements divisible by $p$. 
We refer to Section \ref{realis} for further details.

Now let $M$ be a manifold as in Theorem \ref{main}. 
Let $p$ be an (odd)  prime dividing the order of $\pi_1(M)$, let  $\overline{M} \to M$ be the connected cover corresponding to the inclusion of a Sylow $p$-subgroup $\Gamma \subset \pi_1(M)$ and let $\overline{\phi} : \overline M \to B \Gamma$ be the classifying map.
By Remark  \ref{ex:atoral} \ref{Sylow} the manifold $\overline{M}$ is $p$-atoral, and by construction $\overline{\phi}_*( [\overline{M}]) \in \HH_d(B \Gamma)$ lies in the image of $\Omega^{\SO}_d(B \Gamma) \to \HH_d(B \Gamma)$. 
Using Theorem \ref{representable} the class $\overline{\phi}_*( [\overline{M}]) \in \HH_d(B \Gamma)$ is positive. 
Hence also the class 
 \[
    [\pi_1(M) : \Gamma] \cdot \phi_*([M]) =     \psi_*\big( \overline{\phi}_*( [\overline{M}])\big)  \in \HH_d(B \pi_1(M)) 
 \]
is positive, where $[\pi_1(M) : \Gamma]$ denotes the index of $\Gamma$ in $\pi_1(M)$ and $\psi : B\Gamma \to B\pi_1(M)$ is  induced by the subgroup inclusion $\Gamma \subset \pi_{1}(M)$. 

By the Chinese remainder theorem we find $\alpha_{p} \in \Z$, where $p$ runs through the primes dividing the order of $\pi_1(M) $, such that $ 1 = \sum_{p}  \alpha_{p} \cdot [\pi_1(M) : \Gamma_p]$ where $\Gamma_p \subset \pi_1(M)$ denotes some Sylow $p$-subgroup. 
Hence 
\[
    \phi_*([M]) = \sum_p \alpha_p \cdot [\pi_1(M) : \Gamma_p] \cdot \phi_*([M]) \in \HH^{ \mathscr{Q} ,+}_d(B \pi_1(M); \Z ) \, , 
\]
finishing the proof of Theorem \ref{main}.

We conjecture that Theorem \ref{main} also holds for spin manifolds with vanishing $\alpha$-invariants. 
A proof should be based on real connective $K$-homology instead of ordinary homology;  compare Rosenberg and Stolz \cite{RS}. 
However our homological computations do not  carry over to this case in an obvious way. 
Hence we  leave the spin analogue of Theorem \ref{main} for later investigation.

\medskip

{\em Acknowledgments.} This project was initiated when I was visiting the University of Notre Dame some years ago. 
To Stephan Stolz I owe the idea to study positive scalar curvature metrics on manifolds with Baas-Sullivan singularities for proving Theorem \ref{main}. 
Substantial parts of this research were carried out at the MPI Bonn and the Courant Institute of Mathematical Sciences (NYU). 
The hospitality of the named institutions  is gratefully acknowledged. 
I appreciate a number of helpful suggestions by an anonymous referee, which  led to a significant improvement of the manuscript. 
Many thanks also go to John Bourke from the MSP production team. 

This research has been supported by the Special Priority Programme SPP 2026 {\em Geometry at Infinity}  funded by the DFG.

\section{Review of manifolds with Baas-Sullivan singularities} \label{background_BS} 

We recall some terminology, following mainly \cite{Bunke}*{Section 3.3}, and fix some notation. 
Smooth $d$-dimensional {\em manifolds with corners} $V$ are modeled on subsets  $N(k,U) =  U \times [0,1)^k \subset \R^d$ for $0 \leq k \leq d$, where $U \subset \R^{d-k}$ is open, with smooth transition maps of the form $(x, t_1, \ldots, t_k) \mapsto (x', t_{\sigma(1)}, \ldots, t_{\sigma(k)})$ for some permutation $\sigma$. 
For a precise definition we refer to \cite{Bunke}*{Definition 3.14.} and the subsequent discussion. 
In particular, manifolds with corners are equipped with preferred local collar structures.
\footnote{Some authors use different conventions, compare, for instance, \cite{Joyce}*{Definition 2.2} .} 
The subset  $U \times \{0\} \subset N(k,U)$ defines the {\em points of codimension $k$} in $N(k,U)$. 

Let $V$ be a $d$-dimensional manifold with corners. 
Every point $x \in V^d$ has a  {\em codimension} $0 \leq c(x) \leq d$,  defined with respect to any local chart around $x$. 
This induces a decomposition of $V$ into smooth (in general non-compact) connected submanifolds of $V$, called {\em strata}, of various codimensions. 
Each stratum admits a canonical completion (by adding boundary points to local models), which is itself a manifold with corners, see \cite{Bunke}*{Definition 3.17}. 
The union of strata of codimension at least $1$ in $V$ is denoted by  $\partial V$. 

As usual we require by definition that each $x\in V$ lies in the closure of exactly $c(x)$ codimension-$1$ strata of $V$.
In this case the completions of strata coincide with their respective closures in $V$ (note that this is not true  for the $1$-gon, for example), which are called {\em connected faces} of $V$. 

Manifolds with Baas-Sullivan singularities were  introduced in \cite{baas1}. 
Let us recall some features of the theory which are relevant for our discussion. 
A {\em decomposed manifold} is a  manifold $V$ with corners together with a decomposition 
\[
   \partial V = \partial_0 V  \cup \cdots \cup \partial_n V \, , 
\]
for some $n \in \N$, where each $\partial_i V$ is a disjoint union of connected codimension-$1$ faces of $V$ which is globally collared in $V$ and each connected codimension-$1$ face of $V$ is contained in exactly one $\partial_i V$, see \cite{baas1}*{Definition 2.1}. 
Each $\partial_i V$ has an induced structure of a decomposed manifold by setting $\partial_j ( \partial_i V) = \partial_i V \cap \partial_j V$ for $j \neq i$, and $\partial_j  ( \partial_j V) = \emptyset$ for $0 \leq j \leq n$, compare \cite{baas1}*{page 283}.

\begin{defn} We call the decomposed manifold $\partial_0 V$ the {\em boundary} of $V$. 
If $V$ is compact and $\partial_0 V = \emptyset$, then $V$ is called {\em closed}. 
\end{defn} 

Similar to  \cite{baas1}*{Definition 2.2} we fix a family of closed smooth manifolds $\mathscr{P}= ( P_0 = *, P_1, P_2, \ldots)$, called {\em singularity types}. 
For  $n \in \N$ we set $\mathscr{P}_n= ( P_0, \ldots, P_n)$.
By definition,  a {\em $\mathscr{P}_n$-manifold} is a family of decomposed manifolds 
\[
   A  = ( A(\omega))_{\omega \subset \{0, \ldots, n\}}, 
\]
where $\partial A( \omega) = \partial_0 A(\omega) \cup \cdots \cup \partial_n A(\omega)$,  with $\partial_i A(\omega) = \emptyset$ for $i \in \omega$,  together with  isomorphisms $\partial_i A(\omega) \cong  A(\omega,i) \times P_i$ of decomposed manifolds for $i \in \{ 0, \ldots, n\} \setminus  \omega$. 
\footnote{Some authors use the ``Bockstein'' notation $\beta_{\omega} A$ instead of $A(\omega)$.}
Here we set $\partial_j  ( A( \omega, i) \times P_i) := \partial_j A(\omega,i) \times P_i$ for $0 \leq j \leq n$ and we write $A(\omega, i)$ instead of $A(\omega \cup \{i\})$. 
We also use the shorthand $A$ for the decomposed manifold $A(\emptyset)$.

By definition the following {\em compatibility condition} is required to hold. 
\begin{cond} \label{compcond} 
For all $i, j \notin \omega$ with $i \neq j$  the isomorphisms
\begin{eqnarray*} \label{compa} 
  \partial_iA( \omega) \cap \partial_j A(\omega) = \partial_j (  \partial_i A(\omega) ) & \cong &  \partial_j   ( A(\omega, i) \times P_i) = \partial_j A (\omega, i) \times P_i   \cong A(\omega, i, j) \times P_j \times P_i  \, , \\
  \partial_ i A( \omega) \cap \partial_j  A(\omega)   = \partial_i ( \partial_j A(\omega) ) & \cong & \partial_i    (  A(\omega, j) \times P_j ) = \partial_i A(\omega, j) \times P_j \cong A(\omega, j , i ) \times P_i \times P_j 
\end{eqnarray*} 
coincide after composing one of them with the interchange map $ P_j \times P_i \to P_i \times P_j$.
\end{cond} 

Note that each $\mathscr{P}_n$-manifold $A$ can be regarded as a $\mathscr{P}_{n+1}$-manifold in a canonical way by setting $\partial_{n+1} A := \emptyset$. 
For a $\mathscr{P}_n$-manifold $A$ we define  the  {\em union of singular strata}  of $A$ as
\[
     \sing(A) := \bigcup_{1 \leq i \leq n} \partial_i A \, , 
\]
so that, obviously, $\partial A =   \partial_0 A \cup \sing(A)$ and $ \sing(\partial_0 A) =  \partial_0 A \cap \sing(A)$.

There is a bordism theory $\Omega^{\mathscr{P}_n}_* (-) $, which we call {\em bordism with singularities in $\mathscr{P}_n$} -- compare \cite{baas1}*{page 284 ff.}, where this theory is denoted by $M(\mathscr{P}_n)_*(-)$. 
Given  a  pair of topological spaces $(X,Y \subset X)$ elements in $\Omega^{\mathscr{P}_n}_d(X,Y)$ are by definition  represented by continuous maps $f : A^d \to X$, where (see \cite{baas1}*{Definitions 2.2.~and  2.3.}) 
\begin{enumerate}[label={(\roman*)}] 
   \item $A$ is a compact $d$-dimensional $\mathscr{P}_n$-manifold; 
   \item \label{uksi} on local  models $U \times [0,1)^k $ the map $f$ factors through the projections 
   \[
      U \times [0,1)^k \stackrel{\pr_U}{\longrightarrow}  U;
   \] 
   \item  \label{kaksi} for  all $i \in \{1, \ldots, n\}$ the restriction $f|_{\partial_i A} $ factors as  
   \[
       \partial_i A \cong A(i)  \times P_i \stackrel{\pr_{A(i) }}{\longrightarrow}  A(i)  \to X;
   \]
      \item $f( \partial_0 A ) \subset Y$.  
\end{enumerate}

\begin{defn} \label{mapbs} A continuous map $f : A^d \to X$ with properties \ref{uksi}  and \ref{kaksi} is called {\em compatible} with the singularity structure of $A^d$. 
\end{defn} 

\begin{defn} \label{defbordsing} The homology theory obtained in the limit $n \to \infty$ is called {\em bordism with singularities in $\mathscr{P}$} and denoted by $\Omega^{\mathscr{P}}_*(-)$.
\end{defn} 

There is a straightforward generalization to bordism with tangential structures. 
In this paper we will be working with {\em oriented bordism with singularities} $\Omega^{\SO, \mathscr{P}_n}_*(-)$ with  $n \geq 0$ or $\Omega^{\SO, \mathscr{P}}_*(-)$, where we assume  that 
\begin{itemize} 
    \item all singularity types $P_i$ are even dimensional;
    \item all $P_i$ and $A(\omega)$ are oriented;
   \item \label{orientbla} for $i \notin \omega$ and with respect to the induced orientation on $\partial_i A(\omega)$ (determined by the outward normal) the given isomorphism $\partial_i A(\omega) \cong A(\omega, i) \times P_i$ is orientation preserving,  if an even number of elements in $\omega$ are larger than $i$, and orientation reversing otherwise.  
 \end{itemize}
In a similar way one may define {\em spin bordism with singularities} $\Omega^{{\rm Spin}, \mathscr{P}}_*(-)$, but this theory will not be considered in this paper.

\begin{constr} \label{constr:trafo} For $n \geq 0$ we shall define natural transformations of homology theories
\[
  u  :   \Omega^{\SO, \mathscr{P}_n}_*(-)   \to \HH_*(- ; \Z)  \, . 
\]
Let  $(X,Y)$ be a pair of topological spaces and let $f : A^d \to X$ represent an element in $\Omega^{\SO, \mathscr{P}_n}_d(X,Y)$ with a connected oriented compact $\mathscr{P}_n$-manifold  $A$. 

 Let $A'$ be obtained from $A$ by passing to the quotient space resulting from the identifications $(x, p) \sim (x,p')$ on $\partial_i A = A(i) \times P_i$ for $i = 1, \ldots, n$,  $x \in A(i)$ and $p, p' \in P_i$. 
Intuitively this process may be regarded as ``coning off''  the singularity types $P_1, \ldots, P_n$ in $A$ and thereby introducing genuine singularities.
It is shown by a straightforward computation that  $H_d(A', (\partial_0 A)' ; \Z) \cong \Z$ (recall $\dim P_i \geq 2$ for $i \geq 1$), with a preferred generator $[A', (\partial_0 A)']$ corresponding to  the given orientation of $A$.  
By assumption $f$ factors through a map $f' : ( A', (\partial_0 A)' )  \to (X, Y)$, and we define
\begin{equation} \label{def:u} 
        u( [ f : A \to X ] ) := f'_*( [ A' , (\partial_0 A)' ]) \in \HH_d( X,Y; \Z) \, . 
\end{equation} 
Passing to the limit $n \to \infty$ we also obtain a natural transformation 
\[
  u  :   \Omega^{\SO, \mathscr{P}}_*(-)   \to \HH_*(- ; \Z)  \, . 
\]
\end{constr} 

 By \cite{Nov60} the oriented bordism ring $\Omega^{\SO}_*$ modulo torsion is a polynomial ring. 
There are closed oriented manifolds $Q_1,Q_2,\ldots$ with $\dim Q_i=4i$  such that
\begin{equation*}
\Omega^{\SO}_*/{\text{torsion}}\cong\mathbb{Z}[[Q_1],[Q_2],\ldots] \, , 
\end{equation*}
where $[Q_i] \in \Omega^{\SO}_{4i}$ denotes the bordism class represented by $Q_i$. 
Since $\Omega^{\SO}_*$ contains no odd torsion  \cite{milnorcomplex} the sequence $([Q_i])_{i\geq 1}$ is a regular sequence in $\Omega^{\SO}_* \otimes \Z[1 / 2 ]$. 
Setting $\mathscr{Q} := (Q_0 = *, Q_1, Q_2, \ldots)$  we  arrive at the following fundamental result from \cite{baas1}. 

\begin{prop}  \label{isobordhom} For all $(X,Y)$ the natural transformation $u$ defined in  \eqref{def:u} induces an isomorphism
\[
  u :   \Omega^{SO, \mathscr{Q}}_*(X, Y)\otimes \Z[1/2]  \to \HH_*(X, Y; \Z[1/2]) \, . 
\]
\end{prop}

\begin{cor} \label{iso_odd_order} 
Let $\Gamma$ be a finite group of odd order. 
Then  the map 
\[
  u :   \Omega_*^{\SO, \mathscr{Q}}( B\Gamma)  \rightarrow \HH_*(B \Gamma; \Z) 
\]
is surjective. 
\end{cor} 

\begin{proof} This holds  in degree $0$, when source and target of $u$ are equal to $\Z$. 
Let $d \geq 1$. 
Since $\Gamma$ is of odd order the homology group $\HH_d(B \Gamma; \Z)$ is abelian of odd order. 
Hence for any $m_0 \geq 0$ and any $x \in \HH_d(B \Gamma; \Z)$ there exists $m \geq m_0$ with $2^m \cdot x = x$. 
The claim is hence implied by Proposition \ref{isobordhom} by clearing denominators. 
\end{proof}

In other words: Each homology class in $\HH_*(B \Gamma; \Z)$ is represented by a $\mathscr{Q}_n$-manifold for some $n$ (which has to be larger than $0$ in general, compare Example 7.6). 
Next we will  introduce  and study the notion of positive scalar curvature metrics on these objects.

\section[Positive scalar curvature on such manifolds]{Positive scalar curvature on manifolds with Baas-Sullivan singularities}  \label{Baas_Sul}

\begin{defn} \label{defadmissible} 
An {\em admissible} Riemannian metric on a manifold with corners $V^d$ is a smooth Riemannian metric $g$ on $V$ which on each local model $U \times [0,1)^k $ restricts to a product metric $ g^U  \oplus \eta $. 
Here  and in the following $\eta$ denotes the standard Euclidean metric and $g^U$ is some Riemannian metric on $U \subset \R^{d - k}$. 
\end{defn}

\begin{defn} \label{pos_sing} 
A family of {\em Riemannian singularity types} is a family of singularity types $\mathscr{P} = (P_0 = *, P_1, P_2, \ldots )$ together with Riemannian metrics $h_i$ on $P_i$ for $i \geq 1$. 

We call a family of Riemannian singularity types {\em positive} if each metric $h_i$ for $i \geq 1$  is of positive scalar curvature. 
\end{defn} 

\begin{defn} \label{dist} Let $\mathscr{P}$ be a  family of Riemannian singularity types and let $A$ be a $\mathscr{P}_n$-manifold, possibly with boundary. 
An admissible metric $g$ on $A = A( \emptyset) $ is called {\em $\mathscr{P}$-compatible}  if for each  $\omega \subset \{1, \ldots, n\}$  there is an admissible metric $g(\omega)$ on $A(\omega)$ such that $g = g ( \emptyset)$ and  the metric $g(\omega)$ restricts to the product metric $g(\omega, i ) \oplus h_i$ on $\partial_i A(\omega)  \cong A(\omega, i ) \times P_i$ for $i \in \{ 1, \ldots, n\} \setminus \omega$.
\end{defn}

\begin{lem} \label{compat} Each $\mathscr{P}_n$-manifold admits a $\mathscr{P}$-compatible metric. 
\end{lem} 

\begin{proof} Use  downward induction on the cardinality of $\omega \subset \{1, \ldots, n\}$, starting with $| \omega | = n$. 
\end{proof}

\begin{constr}[Scaling $\mathscr{P}$-compatible metrics] \label{constrscal} 
Let $\mathscr{P}$ be a family of Riemannian singularity types and let $A$ be a $\mathscr{P}_n$-manifold together with a $\mathscr{P}$-compatible  metric $g$.
For  $\lambda > 0$ the scaled metric $\lambda \cdot g$ is not $\mathscr{P}$-compatible unless $\lambda = 1$. 
The following construction will resolve this issue. 

We fix, once and for all,  a smooth cut-off function $\phi: [0,1] \to [0,1]$ equal to $0$ on $[0,1/3]$ and equal to $1$ near $1$. 

Let $\lambda > 0$ and $\delta \geq 3$ real numbers. 
For $\omega \subset \{1, \ldots, n\}$, say $\omega =  ( i_1, \ldots, i_k)$ with $1 \leq i_1 < \cdots < i_k \leq n$, we obtain a $k$-parameter family  $(g_t)_{ t = (t_{i_1}, \ldots, t_{i_k} )  \in [0,\delta]^{k}}$ of Riemannian metrics on  $\partial_{i_1} \cdots \partial_{i_k}  A \cong A(\omega) \times P_{i_1} \times \cdots \times P_{i_k}$   where
 \begin{equation} \label{multpara} 
g_t = \lambda \cdot g(\omega ) \oplus \bigoplus_{i \in \omega}  \Big(  \phi(t_{i} /\delta)  \cdot \lambda   + \big( 1- \phi (t_{i}/\delta ) \big) \Big) h_{i}  \, .  
 \end{equation} 
 
We abbreviate $P_{\omega} := P_{i_1} \times \cdots \times P_{i_k}$. 
With the Euclidean metric $\eta$ on $[0,\delta]^{k}$ we obtain a  smooth Riemannian metric $g_{\omega, \lambda, \delta }$ on $A(\omega)  \times P_{\omega} \times [0, \delta]^{k}$ defined by 
\[
 g_{\omega, \lambda, \delta}  (a, p , t) :=   g_t (a,p) \oplus \eta  \, . 
\]

Choose some monotonically increasing diffeomorphism $\chi: [0,1] \to [0,\delta]$ which has derivative $\sqrt{\lambda}$ near $1$ and denote the induced diffeomorphisms $[0,1]^{k} \to [0,\delta]^{k}$ by $\chi$ as well. 

For $\omega \subset \{1, \ldots, n\}$, if $|\omega| = k$,  then we replace the metric $\lambda \cdot \left( g(\omega) \oplus \bigoplus_{i \in \omega} h_i \oplus \eta \right)$ on the local model $A(\omega) \times P_{\omega} \times [0,1)^{k} \subset A$ by the metric $g_{\omega, \lambda, \delta}$ pulled back along the diffeomorphism $\id \times \id \times \chi : A(\omega) \times P_{\omega} \times [0,1)^{k} \rightarrow A(\omega) \times P_{\omega} \times [0,\delta)^{k}$. 
Continuing with increasing $k=0, \ldots, n$ this construction results in a  smooth metric on $A$.
Furthermore, by  the choice of $\phi$ and since $\delta \geq 3$, there are induced local corner models on $A$ with respect to which this metric is $\mathscr{P}$-compatible. 

This new  metric on $A$ is denoted by $g_{(\lambda, \delta)} $ and is called the {\em $(\lambda, \delta)$-scaling} of $g$. 
 The diffeomorphism $\chi$ and hence the metric $g_{(\lambda, \delta)}$ can be assumed to depend smoothly on $\lambda$ and $\delta$.
  Note that $g_{(\lambda, \delta)}(\omega) = g(\omega)_{(\lambda, \delta)}$ for $\omega \subset \{1, \ldots, n\}$.

For $n = 2$ and $\delta = 3$ the situation is illustrated in Figure~\ref{figure1}, where the  shaded region indicates  the collar near  $\sing (A)$ for the scaled metric $g_{(\lambda, \delta)}$.

\begin{figure} 
\begin{tikzpicture} [scale=1.6]
    \filldraw [gray!15!white]           (-3,2) -- (3.5,2) -- (3.5,-1.5) -- (4, - 1.5) -- (4, 2.5) -- (-3,2.5) -- (-3, 2) ; 
    \draw                    (2.5, 2.5) node[above] {$\lambda g(1, 2) \oplus h_{1} \oplus \lambda h_{2}$}   -- (2.5,1.4) ;
 \draw  (2.5, 1)     node[above] {$\lambda g(1, 2) \oplus \lambda h_{1} \oplus \lambda h_{2}$} -- (4,1) ; 
    \draw[dotted, very thick] (-3,1)  --  (-2.5,1) ;
    \draw[white]                 (2.5, 1) -- node[near end, color=black] {$\lambda g(2) \oplus \lambda h_{2}$} (2.5, -1) ;
  \draw[very thick]     (-2.5,1) -- node[near start, above]  {$\lambda g( 1) \oplus \lambda h_{1}$} (2.5,1)  -- (2.5,-0.3) ;
  \draw[very thick]    (2.5, -0.7) -- (2.5, -1) ;
  \draw[dotted, very thick] (2.5,-1) -- (2.5,-1.5) ;
   \draw[dotted]          (-3,2.5) --  (-2.5,2.5) ; 
  \draw                   (-2.5, 2.5) -- node[near start, above] {$\lambda g( 1) \oplus h_{1}$} (2.5, 2.5) -- (4,2.5) node[above right] {$\lambda g(1, 2) \oplus h_{1} \oplus h_{2}$} -- (4,1) node[above right]  {$\lambda g(1, 2) \oplus \lambda h_{1} \oplus h_{2}$}  -- node[right, near end] {$\lambda g(2) \oplus h_{2}$} (4,-1); 
    
  \draw[dotted]          (4,-1) --   (4,-1.5) ;  
   \filldraw [black]                   (2.5,1) circle [radius=2pt] ;  
   \filldraw [gray!50!white]            (4,2.5)  circle [radius=2pt] (2.5,2.5)  circle [radius=2pt] (4,1)  circle [radius=2pt]   ;  
   \draw node at (0.5 , -0) {$\lambda g(\emptyset)$} ; 
    \end{tikzpicture} 
\caption{$\mathscr{P}_2$-manifold $A$ with scaled metric $g_{(\lambda, \delta)}$} 
\label{figure1}
\end{figure}
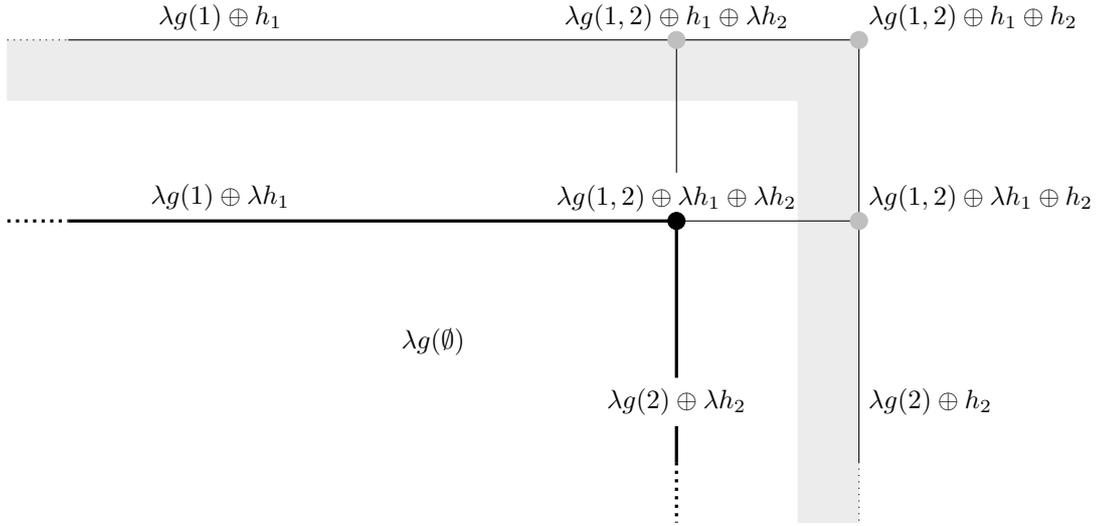 
\end{constr}

\begin{defn} \label{singupos} Let $\mathscr{P}$ be a family of Riemannian singularity types, let $A$ be a $\mathscr{P}_n$-manifold and let $g$ be a $\mathscr{P}$-compatible metric on $A$.
We say that $g$  is   {\em singularity-positive}, if for all $1 \leq i \leq n$ the product metric $g(i) \oplus h_i$ on $\partial_i A = A(i) \times P_i$ is of positive scalar curvature. 
\end{defn} 

\begin{prop} \label{standard_riem} Let $\mathscr{P}$ be positive, let $A$ be a compact $\mathscr{P}_n$-manifold, and let $g$ be a $\mathscr{P}$-compatible metric on $A$. 
Then there exists $\lambda \geq 1$ and $\delta_0 \geq 3$ such that  for all $\delta \geq \delta_0$ the metric $g_{(\lambda, \delta)}$ is singularity positive.
\end{prop}

\begin{proof} 
Since the metrics $h_1, \ldots, h_n$ are of positive scalar curvature and $A$ is compact we find  some  $\lambda \gg 1$ such that  the metric $\lambda \cdot g(i) \oplus h_i $ on $\partial_i A = A(i) \times P_i$ is of positive scalar curvature for all $1 \leq i \leq n$. 

By the additivity of scalar curvature in Riemannian products and since $\lambda \geq 1$ the metric $g_{t}$  in \eqref{multpara} is of positive scalar curvature whenever $t_{i_j} \leq \delta/3$ for some $1 \leq j \leq k$. 

For $\omega \subset \{1, \ldots, n\}$ with $|\omega| = k$,  we obtain Riemannian submersions
\[
  \Big(       A( \omega)  \times P_{\omega} \times  [0,\delta ]^{k}  \; ,  \;    g_{\omega, \lambda, \delta} \Big)  \longrightarrow \Big( [0,\delta]^{k},  \eta  \Big) \, , 
\]
whose fibers are equipped with the metrics $g_{t}$. 

By the O'Neill formula for the scalar curvature in Riemannian submersions   \cite{Besse}*{(9.37)},  we find $\delta_0 \geq 3$ such that for all $\delta \geq \delta_0$  the metric $ g_{\omega, \lambda, \delta}$ is of positive scalar curvature on the subset 
\[
   \{   (a , p, t) \in A(\omega) \times P_{\omega} \times[0,\delta]^{k}  \mid 0 \leq  t_{i_j}   \leq \delta /3  \textrm{  for some }  1 \leq j \leq k \} \, .
\]
This implies the assertion of Proposition \ref{standard_riem}. 
 \end{proof} 

We need the following variation of Definition \ref{singupos}.

\begin{defn} \label{comppos}  Let $\mathscr{P}$ be a family of Riemannian singularity types, let $A$ be a $\mathscr{P}_n$-manifold and let $g$ be a $\mathscr{P}$-compatible metric on $A$.
We say that $g$ is  {\em positive}, if for all $\omega \subset \{  1 , \ldots, n\}$ \rot{(including $\omega = \emptyset$)}  the metric $g(\omega)$ on $A(\omega)$ is of positive scalar curvature. 
\end{defn} 

This condition, which is stronger than just requiring the metric $g$ on $A$  to be of positive scalar curvature, will become important  in the proof of the next proposition.
Note that positive metrics are singularity positive in the sense of Definition \ref{singupos}.

\begin{prop} \label{scalescal}  Let $\mathscr{P}$ be a family of singularity types and $A$ be a compact $\mathscr{P}_n$-manifold  together with a $\mathscr{P}$-compatible  positive metric $g$.
Let $\Lambda \subset (0, \infty)$ be a compact subset and let $s > 0$. 
\begin{enumerate}[label={(\roman*)}] 
  \item \label{uno} Let $\mathscr{P}$ be positive. 
      Then there exists $\delta_0 \geq 3$ such that for all  $\lambda \in \Lambda$ and  $\delta \geq \delta_0$ the scaled metric $g_{(\lambda, \delta)}$ is positive. 
   \item \label{duo} There exists $0 < \lambda_0 \leq 1$ such that for all $0 < \lambda \leq  \lambda_0$ there exists $\delta_0 \geq 3$ such that for all $\delta \geq \delta_0$ and $\omega \subset \{ 1, \ldots, n\}$ we have $\scal_{g_{(\lambda, \delta)}(\omega)}  > s$. 
 \end{enumerate}  
\end{prop}

\begin{proof} 
For $\omega \subset \{1, \ldots, n\}$ the metric $g(\omega)$ is of positive scalar curvature by assumption and hence  \eqref{multpara} implies, assuming positivity of $\mathscr{P}$ in case (a): 
\begin{enumerate}[label={(\alph*)}] 
   \item For all $\lambda \in \Lambda$ and $t \in [0,\delta]^{k}$ we have $\scal_{g_{t}} > 0$. 
  \item  There exists $0 < \lambda_0 \leq 1$ such that for all $0 < \lambda \leq \lambda_0$ and $t \in [0,\delta]^{k}$ we have $\scal_{ g_{t} }> s$. 
\end{enumerate} 
Using the O'Neill formula and the compactness of $\Lambda$ this implies, on $A = A( \emptyset)$: 
 \begin{enumerate}[label={(\alph*)}]
  \item \label{aa} There exists $\delta_0 \geq 3$ such that for $\lambda \in \Lambda$ and $\delta \geq \delta_0$ we have $\scal_{g_{(\lambda, \delta)}(\emptyset)} > 0$.
  \item \label{bb} There exists $0 < \lambda_0 \leq 1$ such that for all $0 < \lambda \leq \lambda_0$ there exists  $\delta_0 \geq 3$  such that for all $\delta \geq \delta_0$ we have $\scal_{g_{(\lambda, \delta)}(\emptyset)} > s$. 
  \end{enumerate}

Now, for any $\theta \subset \{1, \ldots, n\}$, a similar argument applies to $A(\theta)$   instead of $A = A(\emptyset)$  so that we can pass to the maximum  of the resulting constants $\delta_0$ in \ref{aa}, and  to the minimum of the resulting constants $\lambda_0$ and the maximum of the resulting constants $\delta_0$  in \ref{bb}, in order to prove the required lower estimates of scalar curvatures on all $A(\theta)$. 
\end{proof}

\begin{cor} \label{add:restrict} 
Let $\mathscr{P}$ be positive, let $A$ be a compact $\mathscr{P}_n$-manifold, and let $g$ be a $\mathscr{P}$-compatible metric on $A$.
Furthermore, let $C \subset \partial_0 A$ be a union of components of $\partial_0 A$  and assume that  the restriction of $g$ to $C$ is positive (in the sense of Definition \ref{comppos}). 

Then there exists $\lambda \geq 1$ and $\delta_0 \geq 3$ such that  for all $\delta \geq \delta_0$ the scaled metric $g_{(\lambda, \delta)}$ is singularity positive and still restricts to a positive metric on $C$.
\end{cor} 

\begin{proof} Let $\lambda$ and $\delta_0$ be chosen as in Proposition \ref{standard_riem}. 
The claim follows from Proposition \ref{scalescal} \ref{uno} applied to $A := C$ and $\Lambda := \{\lambda\}$, possibly after passing to some larger $\delta_0$. 
\end{proof} 

We can now show the following bordism principle. 

 \begin{prop}  \label{bordprinc} Let $\mathscr{P}$ be a positive family of singularity types and let $V$ be a compact $\mathscr{P}_n$-manifold with $\dim V \geq 6$.
 Assume that the boundary $\partial_0 V $ decomposes as a disjoint union $\partial_0 V =  A \sqcup M$, where $A$ is a closed $\mathscr{P}_n$-manifold equipped with a $\mathscr{P}$-compatible positive metric, and $M$ is a closed smooth manifold. 
Furthermore assume that the inclusion $M \hookrightarrow V$ is a $2$-equivalence. 

Then $M$ carries a Riemannian metric of positive scalar curvature. 
 \end{prop} 

\begin{proof} By Corollary \ref{add:restrict}  we find a $\mathscr{P}$-compatible singularity-positive metric $g$ on $V$ which restricts to a positive metric  on   $A \subset \partial_0 V$.

For $ 1\leq \ell \leq k \leq n+1$ we consider the face 
\[
   \partial_{\ell} [0,1]^k : = \{ (t_1, \ldots, t_{k}) \in [0,1]^k \mid t_{\ell} = 1 \} \subset [0,1]^k \, . 
\]
Each $\partial_{\ell} [0,1]^k$ can be identified with $[0,1]^{k-1}$ in a canonical way and $\partial_{\ell} [0,1]^k$ is equipped with a collar of width $0.1$ equal to 
\[
   \partial_{\ell} [0,1]^k  \times (0.9, 1] = \{ (t_1, \ldots, t_{k}) \in [0,1]^k \mid  0.9 < t_{\ell} \leq 1 \} \subset [0,1]^k . 
\]
For  $1 \leq k  \leq n+1$ we fix smooth hypersurfaces $\mathscr{H}^{k-1} \subset [0,1]^k$ homeomorphic to compact $(k-1)$-balls subject to the following conditions: 
\begin{itemize} 
    \item   $\mathscr{H}^0 = \{1 / 2 \} \subset [0,1]$.  
    \item $\mathscr{H}^{k-1}$ is invariant under permutations 
    \[
       [0,1]^k \to [0,1]^k ,  \quad (t_1, \ldots, t_k) \mapsto (t_{\sigma(1)}, \ldots, t_{\sigma(k)}). 
    \]  
     \item  For $2 \leq k \leq n+1$ the hypersurface $\mathscr{H}^{k-1}$ is of product form in the collar neighborhood of width $0.1$ of each codimension $1$ face $\partial_{\ell} [0,1]^{k}  \subset [0,1]^{k}$ for $1 \leq \ell \leq k$,  and meets this face in $\mathscr{H}^{k-2}$; 
   \item  the metric $\gamma_{k-1}$ on $\mathscr{H}^{k-1}$ induced from the Euclidean metric $\eta$ on $[0,1]^k$ is of nonnegative scalar curvature.
\end{itemize} 

One explicit construction of $\mathscr{H}^{k-1}$ is by attaching a $C^1$-collar  of width $1/5$ to the shifted spherical segment $ (4/5, \ldots , 4/5) - \{ t  \in [0,1]^k \mid \| t \| = 3/10 \} \subset [0,1]^k $ and smoothing. 
 
Replacing  $U \times [0,1)^k$ by $U \times \mathscr{H}^{k-1} $ in local models of $V$ for increasing $1 \leq k \leq n+1$ we obtain a smooth hypersurface $\partial W \subset V$ contained in the collar neighborhood of $\partial V$, where we recall that $\partial V$ is the set of points of codimension at least $1$ in $V$. 
 (We write $\mathscr{H}^{k-1}$ for $\mathscr{H}^{k-1} \cap [0,1)^k$.)
 The  hypersurface $\partial W$ is the boundary of a smooth embedded codimension zero submanifold $W$ of $V$, which we may think of $V$ with ``smoothened corners''. 

We obtain a decomposition $\partial W = C_0 \sqcup C_1$  where $C_0$ and $C_1$ are disjoint smooth submanifolds of $\partial W$ with $C_1 = M$. 
Furthermore $C_1 \hookrightarrow W$ is a $2$-equivalence. 

We claim that the smooth manifold $C_0$ carries a Riemannian metric of positive scalar curvature, such that Theorem \ref{bordprinc} follows from the usual bordism principle for positive scalar curvature metrics, see \cite{Stolz}*{Extension Theorem 3.3}. 

By assumption the induced metrics on the local models  $ V(\omega) \times \prod_{i \in \omega} P_i \times [0,1)^{k}$ of $V$ for $\omega \subset \{0, \ldots, n\}$, $\omega \cap \{1, \ldots, n\} \neq \emptyset$ with $|\omega| = k$ are of product form $ g(\omega) \oplus \bigoplus_{i \in  \omega} h_i \oplus \eta$  (here we set $h_0 = 0$) and of positive scalar curvature, as $g$  is singularity-positive. 
Furthermore the metric $g$ is of positive scalar curvature in the collar neighborhood $A \times P_0 \times [0,1) =  A \times [0,1)$, as  $g$ restricts to a positive metric on $A$. 

Since the metrics $\gamma_{k-1}$ on $\mathscr{H}^{k-1}$ have nonnegative scalar curvature this implies that the restricted metrics $ g(\omega) \oplus \bigoplus_{i \in  \omega} h_i \oplus \gamma_{k-1} $ are of positive scalar curvature on $ V(\omega)  \times \prod_{i \in \omega} P_i \times \mathscr{H}^{k-1}$  for these $\omega$  as well as   on $A \times  \mathscr{H}^0  =  A \times  \{1/2 \}$. 

Altogether we obtain a positive scalar curvature metric on $C_0$ as required. 
\end{proof} 

Let  $\mathscr{Q} := ( Q_0 = *, Q_1, Q_2, \ldots )$  be a family of singularity types as in Proposition \ref{isobordhom}. 
For $i \geq 1$ we can assume that $Q_i$ is equipped with a positive scalar curvature metric $h_i$ -- compare \cite{GL} -- such that $\mathscr{Q}$ is a  positive family of singularity types in the sense of Definition \ref{pos_sing}.

\begin{defn} \label{poshom} Let $X$ be a topological space.
A homology class $h \in \HH_d(X; \Z)$ is called \emph{positive} with respect to $\mathscr{Q}$, if there is a bordism class $ [f : A^d \to X] \in \Omega _d^{\SO, \mathscr{Q}}(X)$ with the following properties: 
\begin{itemize} 
   \item $A$ admits a $\mathscr{Q}$-compatible positive metric (see Definition \ref{comppos}).   
   \item $u ( [f : A^d \to X]) = h$, where $u : \Omega _d^{\SO, \mathscr{Q}}(X) \to \HH_d(X; \Z)$ is as defined in Construction \ref{constr:trafo}. 
\end{itemize} 
The subgroup of all positive homology classes with respect to $\mathscr{Q}$ is denoted by $\HH_d^{\mathscr{Q}, +}(X; \Z)$.

Note that a priori the subgroup of positive homology classes depends on the choice of the metrics $h_i$, and that 
positive homology is functorial in that a map $X \to Y$ of topological spaces induces a map $\HH_*^{\mathscr{Q}, +}(X; \Z) \to \HH_*^{\mathscr{Q}, +}(Y; \Z)$.  
\end{defn} 

\noindent{\em Proof of Theorem \ref{maintwo}.} 
\rot{First assume that $M$ is equipped with a positive scalar curvature metric $g$. 
Regarding $M$ as a Baas-Sullivan manifold with no singular strata, $g$ is a positive metric on $M$ in the sense of Definition \ref{comppos}. 
Hence $\phi_*([M]) \in \HH^{\mathscr{Q}, +}_d(B \pi_1(M) ; \Z)$ as required.} 

For the other implication assume   $\phi_*([M]) \in \HH^{\mathscr{Q}, +}_d(B \pi_1(M); \Z)$. 
We write $\phi_*([M]) = u( [ f : A^d \to B \pi_1(M) ])$ where $A$ is equipped with a $\mathscr{Q}$-compatible positive metric. 

Using an inclusion $* \rightarrow B \pi_1(M)$ the manifold $M$ represents a class 
\[
   [M] \in \Omega^{\SO}_d( B \pi_1(M)). 
\] 
Then $\beta :=  [ \phi : M \to B \pi_1(M) ] - [M] \in \tilde{\Omega}_d^{\SO}(B \pi_1(M) )$, the reduced oriented bordism group of $B \pi_1(M)$. 
Since $\tilde{\Omega}_d^{\SO}(B \pi_1(M) )$ is a finite abelian group of odd order by assumption on $\pi_1(M)$ and by the Atiyah-Hirzebruch spectral sequence, we find,  for each $m_0 \geq 0$, an $m \geq m_0$ with $2^m \cdot \beta = \beta$. 

Each element in the kernel of the map 
\[
     u :   \Omega^{\SO, \mathscr{Q}}_d(B \pi_1(M) ) \to \HH_d( B \pi_1(M) ; \Z)
\]
in  Corollary \ref{iso_odd_order} is $2$-power torsion by Proposition \ref{isobordhom}, and hence, using $d > 0$, there is some $m_0 \geq 0$ with 
\[
   2^{m_0} \cdot \big(  [ f : A^d \to B \pi_1(M) ] - \beta \big) = 0 \in  \Omega^{\SO, \mathscr{Q}}_d(B \pi_1(M) ) \, . 
\]
Hence there is an $m \geq m_0$ with
\begin{eqnarray} \label{multiple} 
    2^m \cdot   [ f : A^d \to B \pi_1(M) ] = \beta =  [ \phi : M \to B \pi_1(M) ] - [M] \in \Omega^{\SO, \mathscr{Q}}_d(B \pi_1(M) ) \, . 
\end{eqnarray} 
Since $d \geq 5$ we can represent $[M] \in \Omega_d^{\SO}$ by a closed oriented smooth $d$-manifold $N$ with a positive scalar curvature metric by \cite{GL}*{Corollary C}.  
By \eqref{multiple} there exists a compact connected oriented $\mathscr{Q}$-bordism $V \to B \pi_1(M) $  between $M \sqcup \overline{N} \to B \pi_1(M) $ and $\coprod_{2^m}  ( f : A \to B \pi_1(M))$. 
Here $\overline{N}$ denotes $N$ with the reversed orientation.

We can assume that the inclusion $M \hookrightarrow V$  is a $2$-equivalence by applying surgeries to the interior of $V$. 
For this we observe that the induced homomorphism $\pi_1(V) \to \pi_1(B \pi_1(M) )$ is surjective and has finitely generated kernel since $\pi_1(V)$ is finitely generated and $\pi_1(M)$ is finite by assumption. 
This kernel can hence be killed by surgeries along finitely many embedded circles in the interior of $V$ with trivial normal bundles, thus achieving $\pi_1(M) \cong \pi_1(V)$. 
Since now $\pi_1(V)$ is finite and $V$ is compact, $\pi_2(V)$ is finitely generated and so is the cokernel of $\pi_2(M) \to \pi_2(V)$. 
Moreover each element in this cokernel can be represented by an embedded $2$-sphere in the interior of $V$ with trivial normal bundle, the universal cover of $M$ being nonspin since $M$ is nonspin and $\pi_1(M)$ is of odd order. 
We can hence apply finitely many surgeries to the interior of $V$  to make $\pi_2(M) \to \pi_2(V)$ surjective, thus achieving our goal.

Now the assertion of  Theorem \ref{maintwo} follows from Proposition \ref{bordprinc}.

\begin{rem} The language developed in this section allows an alternative approach to the results in \cite{F}. 
\end{rem} 

\section{Admissible products} \label{admissible_prod} 

The cartesian product of two manifolds $A$ and $B$ with corners carries an induced structure of a manifold with corners.
However, the construction of the product of $\mathscr{P}_n$-manifolds as a $\mathscr{P}_n$-manifold is more involved.

In order to illustrate the issue  let $A$ and $B$ be smooth manifolds with boundaries diffeomorphic to the closed manifold $P_1$. 
This induces the structure of  $\mathscr{P}_1$-manifolds on $A$ and $B$ where $A(1) = B(1) = \{ *\}$. 
We obtain $\partial (A \times  B) = ( P_1   \times B)  \, \cup  \, ( A \times  P_1  )$,  but this does not induce the structure of a $\mathscr{P}_1$-manifold on $A \times B$ (even after straightening the $\pi/2$-angle at $\partial A \times \partial B$), since the $P_1$-factors on the two pieces of $\partial (A \times  B)$ correspond to different $P_1$-factors in the intersection  $( P_1 \times B)  \, \cap  \, ( A \times P_1)  = P_1 \times P_1$. 
Therefore an additional construction is required, which, roughly speaking, interchanges these two factors at the glueing region. 

This problem was discussed in  \cites{Bot92, Mironov, Morava, Shimada}, resulting in an obstruction of order at most $2$ if  $P_1$ is of even dimension. 
In the  following we present an explicit geometric construction, which somewhat differs from the mentioned sources and is well adapted to our purpose.
We will work in an oriented setting and in particular assume that all singularity types $P_i$ for $i \geq 1$ are even dimensional. 

In the following we fix $n \geq 0$. 
Let $A$ and $B$ be $\mathscr{P}_n$-manifolds with decompositions
\[
    \partial A   =  \partial_0 A \cup \cdots  \cup \partial_n A \, , \quad  \partial B  =  \partial_0 B \cup \cdots \cup \partial_n B .
\]
This includes the case that $\partial_i A = \emptyset$ or $\partial_i B = \emptyset$ for some $i = 0, \ldots, n$. 
 In particular, $A$ or $B$ are allowed to be smooth manifolds without singular strata.

In the remainder of the construction we fix a two dimensional compact hexagonal manifold $\mathfrak{X}$ with corners,  see the dark grey region in Figure \ref{attach}.

For $\omega \subset \{ 1, \ldots, n\}$ we will construct a manifold with corners $A \times_{\omega} B$, which, intuitively speaking,  is the cartesian product $A \times B$ with all codimension $2$-singularities $\partial_i A \times \partial_i B = (A(i) \times B(i)) \times P_i \times P_i$ for $i \in \omega$  resolved. 
The construction runs by induction on the cardinality of $\omega$.

For $\omega = \emptyset$ we set $A \times_{\omega} B := A \times B$, the cartesian product of $A$ and $B$ with its induced structure of a manifold with corners. 
In addition we smoothen the $\pi/2$-angle appearing at $\partial_0 A \times \partial_0 B$. 

Assume that $1 \leq \ell \leq n$ and $A \times_{\omega} B$ has been constructed whenever $| \omega | = \ell - 1$.
Let  $\omega \subset \{1, \ldots, n\}$ with $| \omega| = \ell$. 

Choose some $i \in \omega$ and consider the collar neighborhood  
\[
 \big(  \partial_{i} A \times [0,1) \big) \times_{\omega \setminus \{i\} } \big( \partial_{i} B \times [0,1) \big) = \big(  \partial_{i} A \times_{\omega \setminus \{i\} }  \partial_{i} B \big)  \times [0,1)^2 \subset A \times_{\omega \setminus \{i\} } B
\]
of the codimension-$2$ face $ \partial_{i} A \times_{\omega \setminus \{i\} } \partial_{i} B \subset A \times_{\omega \setminus \{i\} } B$. 
The manifold $A \times_{\omega} B$ is obtained by removing this collar neighborhood from two disjoint copies of $A \times_{\omega \setminus \{i\} } B$ and gluing in the handle $( \partial_{i} A \times_{\omega \setminus \{i\} } \partial_{i} B ) \times \mathfrak{X}$ as  indicated in Figure \ref{attach}, where $\mathfrak{X}$ is drawn in dark grey color.
The factor $P_{i} \times P_{i}$ appearing in 
 \[
      \partial_{i} A \times_{\omega \setminus \{i\} } \partial_{i} B = (A(i) \times_{\omega \setminus \{i\} } B(i)) \times P_{i} \times P_{i} 
\]
is glued  to  the left hand copy of $\big( A \times_{\omega \setminus \{i\} } B \big) \setminus \Big( \big(  \partial_{i} A \times_{\omega \setminus \{i\} }  \partial_{i} B \big)  \times [0,1)^2 \Big)$ by the identity map, and to the right hand copy  by the interchange map $(p_1, p_2) \mapsto (p_2, p_1)$.

The interchange  map $P_{i} \times P_{i}  \to P_{i}  \times P_{i}$ is orientation preserving, since  $P_{i}$ is even dimensional, and hence the manifold  $A \times_{\omega} B$ carries an induced orientation.

\begin{rem} 
\begin{enumerate}[label={(\roman*)}] 
   \item \label{always} If $\partial_i A = \emptyset$ or $\partial_i B = \emptyset$, then  $A \times_{\omega} B$ consists of two disjoint copies of $A \times_{\omega \setminus \{i\} } B$.  
   \item The manifold $A \times_{\omega} B$ does not depend on the choice of $i \in \omega$, up to canonical diffeomorphism.
\end{enumerate} 
\label{admprodind} 
\end{rem} 

For $i \in \omega$ we set 
\[
    \partial_i (A \times_{\omega} B) :=  2 \cdot \big( (\partial_i A \times_{\omega \setminus \{i\} } B) \cup_{\partial_{i} A \times_{\omega \setminus \{i\} } \partial_{i} B} ( A \times_{\omega \setminus \{i\} } \partial_{i} B) \big) \, , 
   \]
 where  the two copies on the right hand side correspond to the upper and  lower thick boundary pieces in Figure \ref{attach}.
Notice that 
\[
  (\partial_i A \times_{\omega \setminus \{i\} } B) \cap  ( A \times_{\omega \setminus \{i\} } \partial_{i} B) = \partial_{i} A \times_{\omega \setminus \{i\} } \partial_{i} B = ( A(i) \times_{\omega \setminus \{i\} } B(i) ) \times P_{i} \times P_{i}
 \]
 and that the identification along this subspace  interchanges the two factors in $P_{i} \times P_{i}$, thus realizing our initial goal.

In particular we get an induced isomorphism
\[
  \partial_{i} (A \times_{\omega} B) \cong (A \times_{\omega} B)(i) \times P_{i} 
\]
where 
\begin{equation} 
  (A \times_{\omega} B) ( i) := 2 \cdot \big( (A \times_{\omega \setminus \{i\} }  B(i)) \cup_{( A(i) \times_{\omega \setminus \{i\} } B(i)) \times  P_i } (A(i) \times_{\omega \setminus \{i\} }  B  ) \big)  \, . 
\label{inside} 
\end{equation} 

This concludes the induction step.

 \begin{figure} 
\begin{tikzpicture} [scale=0.9] 
   \filldraw[gray!10!white]   (-5, 1) -- (-4,0) -- (-2.5, 1.5) -- (-3.5, 2.5) -- (-5, 1) ;
   \filldraw[gray!10!white]    (-5,-1) -- (-4,0) -- ( -2.5, -1.5) -- (-3.5, -2.5) -- ( -5, -1) ; 
    \filldraw[gray!10!white]   (5, 1) -- (4,0) -- (2.5, 1.5) -- (3.5, 2.5) -- (5, 1) ;
   \filldraw[gray!10!white]    (5,-1) -- (4,0) -- (2.5, -1.5) -- (3.5, -2.5) -- ( 5, -1) ; 
   \filldraw[gray!30!white]    (-4,0) -- (-2.5, -1.5) to[out=45, in=135] (2.5, -1.5) -- (4,0) -- (2.5, 1.5) to[out=225, in=315] (-2.5, 1.5) -- (-4,0) ; 
  \draw[thick]  (-3,2) -- (-2.5, 1.5) ; 
  \draw[thick] (-3, -2) -- (-2.5, -1.5) ; 
  \draw[thick]  (3,2)  -- (2.5, 1.5) ;
  \draw[thick]  (3,-2) -- (2.5, -1.5) ; 
  \draw  (-2.5,1.5) -- (-1,0) -- (-2.5, -1.5) ;
  \draw  (2.5,1.5)  -- (1,0) --  (2.5,-1.5) ;
  \draw[dotted , thick] (-3.5, 2.5) --  (-3,2) ; 
  \draw[dotted,  thick] ( -3.5, -2.5) -- (-3, -2) ; 
  \draw[dotted,  thick] ( 3,2) -- (3.5, 2.5) ; 
  \draw[dotted,  thick] ( 3,-2) -- ( 3.5, -2.5) ; 
  \draw[dotted] (-2.5,1.5) -- (-4, 0)  -- (-2.5, -1.5) ;
  \draw[dotted] (2.5,1.5) -- ( 4,0) -- (2.5, -1.5) ;
  \draw[->] (0, 3.5) node   [above] {$  ( \partial_{i} A \times_{\omega \setminus \{i\} } \partial_{i} B ) \times [0,1)^2 $} -- (-2.5,0.5) ; 
   \draw[->] (0, -3.5) node   [below] {$  ( \partial_{i} A \times_{\omega \setminus \{i\}} \partial_{i} B ) \times \mathfrak{X}  $} -- (0,-0.2) ; 
  \draw[->] (0,3.5) -- (2.5, 0.5) ;
  \draw[->] (-6,2) node [above] {$( \partial_{i} A \times_{\omega \setminus \{i\}} B )  \times [0,1)$} -- (-4,1) ;
   \draw[->] (-6,-2) node [below] {$( A \times_{\omega \setminus \{i\}} \partial_{i} B ) \times [0,1)$} -- (-4,-1) ; 
    \draw[->] (6,2) node [above] {$( A \times_{\omega \setminus \{i\}} \partial_{i} B)  \times [0,1)$} -- (4,1) ;
   \draw[->] (6,-2) node [below] {$( \partial_{i} A  \times_{\omega \setminus \{i\}} B  ) \times [0,1)$} -- (4,-1) ; 
   \draw[thick]  (-2.5,1.5)  to[out=-45, in=225]  (2.5,1.5) ;
   \draw[thick]  (-2.5, -1.5) to[out=45, in=135] (2.5, -1.5) ; 
   \end{tikzpicture} 
  \caption{Construction of admissible products}
  \label{attach} 
 \end{figure}
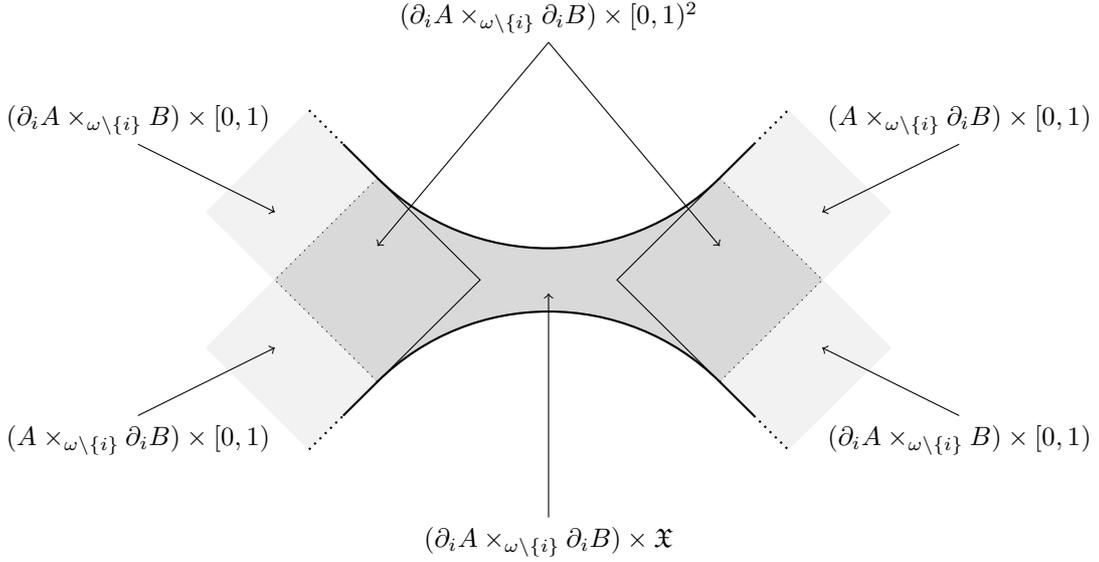

\begin{defn} \label{admissibleprod} The manifold $A \ttimes B := A \times_{\{ 1, \ldots, n\} } B$ is called the {\em admissible product} of $A$ and $B$. 
\end{defn} 

\begin{prop} \label{inducedstr} 
The admissible product $A \ttimes B$ carries an  induced structure of a $\mathscr{P}_n$-manifold.
\end{prop} 

\begin{proof} By construction $A \ttimes B$ carries the structure of a manifold with corners (with respect to appropriate local models) and is a decomposed manifold with decomposition $\partial ( A \ttimes B) = \partial_0(A \ttimes B) \cup \cdots \cup \partial_n (A \ttimes B)$, where we set 
\[
    \partial_0 ( A \ttimes B) := ( \partial_0 A \ttimes B)  \cup_{\partial_0 A \ttimes \partial_0 B} ( A \ttimes \partial_0 B) \, .
\]
(Recall the smoothening of the $\pi/2$-angle at $\partial_0 A \times \partial_0 B$ at the initial stage of the inductive construction). 

It remains to define the decomposed manifolds $(A \ttimes B) ( \omega)$ for $\omega \subset \{1, \ldots, n\}$ in such a way that the compatibility Condition  \ref{compcond} for decomposed manifolds holds.
First we study the case when $\omega$ has two elements.

Let $i, j \in \{ 1, \ldots, n\}=:[n]$ with $i \neq j$.  
By   \eqref{inside} we have
\begin{align*} 
   \partial_{j} \partial_i (A \ttimes B) & = 2 \cdot \Big( \partial_j \big(A \times_{[n]\setminus\{i\} }  B(i) \big) \cup_{\left( \partial_j ( A(i) \times_{[n] \setminus\{i\} } B(i))\right) \times  P_i } \partial_j \big(A(i) \times_{[n]\setminus\{i\} }  B  ) \Big) \times P_i \, , \\
          & = 2 \cdot \Big( \big( A \times_{[n]\setminus\{i\} }  B(i)\big)(j) \cup_{\left(  A(i) \times_{[n] \setminus\{i\} } B(i)\right)(j)  \times  P_i } \big(A(i) \times_{[n]\setminus\{i\} }  B  \big)(j)  \Big) \times P_j \times P_i \, , 
\end{align*} 
and likewise
\begin{align*} 
   \partial_{i} \partial_j (A \ttimes B) & = 2 \cdot \Big( \partial_i \big(A \times_{[n]\setminus\{j\} }  B(j)\big) \cup_{\left( \partial_i ( A(j) \times_{[n] \setminus\{j\} } B(j))\right) \times  P_j } \partial_i \big(A(j) \times_{[n]\setminus\{j\} }  B  \big) \Big) \times P_j \, , \\
          & = 2 \cdot \Big( \big( A \times_{[n]\setminus\{j\} }  B(j)\big)(i) \cup_{\left( A(j) \times_{[n] \setminus\{j\} } B(j)\right)(i)   \times  P_j } \big(A(j) \times_{[n]\setminus\{j\} }  B  \big)(i)  \Big) \times P_i \times P_j \, . 
\end{align*} 
Furthermore 
\begin{align*} 
\big(A \times_{[n]\setminus\{i\} }  B(i) \big)(j)  & = 2 \cdot \Big( \big( A  \times_{[n]\setminus \{ i,j\} }  B(i,j)\big) \cup_{\left( A(j) \times_{[n]\setminus \{ i,j\}} B(i,j) \right) \times P_j}   \big( A(j) \times_{[n]\setminus\{i,j\}} B(i) \big) \Big)  \, , \\
 \big(A(i) \times_{[n]\setminus\{ i \}} B \big)(j) & = 2 \cdot \Big( \big( A(i)  \times_{[n]\setminus \{ i,j\}} B(j)\big) \cup_{\left( A(i,j) \times_{[n]\setminus \{ i,j\}} B(j) \right) \times P_j} \big( A(i,j ) \times_{[n] \setminus \{i,j\}} B \big) \Big)\, , 
 \end{align*} 
 and likewise 
 \begin{align*} 
\big(A \times_{[n]\setminus\{j\} }  B(j)\big)(i)  & = 2 \cdot \Big( \big( A  \times_{[n]\setminus \{ i,j\} }  B(i,j)\big) \cup_{\left( A(i) \times_{[n]\setminus \{ i,j\}} B(i,j) \right) \times P_i}   \big( A(i) \times_{[n]\setminus\{i,j\}} B(j) \big) \Big) \, , \\
 \big(A(j) \times_{[n]\setminus\{ j \}} B \big)(i) & = 2 \cdot \Big( \big( A(j)  \times_{[n]\setminus \{ i,j\}} B(i)\big) \cup_{\left( A(i,j) \times_{[n]\setminus \{ i,j\}} B(i) \right) \times P_i} \big( A(i,j ) \times_{[n] \setminus \{i,j\}} B \big) \Big) \, . 
 \end{align*}   
Similar computations apply to $\big(A(i) \times_{[n]\setminus\{i\} }  B(i)  \big)(j) $ and $\big(A(j) \times_{[n]\setminus\{j\} }  B(j)  \big)(i) $. 

Defining  $( A \ttimes B) (i,j)$ as  
\[
 4 \cdot \Big( \big( A \times_{[n]\setminus \{ i,j\}} B(i,j) \big) \cup \big( A(i) \times_{[n]\setminus\{i,j\}} B(j)\big)  \cup \big( A(j) \times_{[n]\setminus\{i,j\}} B(i) \big)  \cup  \big( A(i,j ) \times_{[n]\setminus\{i,j\}} B \big)  \Big)
\]
where we glue 
\begin{itemize} 
   \item $A \times_{[n]\setminus \{ i,j\}} B(i,j)$ and $  A(i) \times_{[n]\setminus\{i,j\}} B(j)$ along $\left( A(i) \times_{[n] \setminus \{ i,j\} } B(i,j) \right) \times P_i$,
   \item $A \times_{[n]\setminus \{ i,j\}} B(i,j)$ and $  A(j) \times_{[n]\setminus\{i,j\}} B(i)$ along $\left(A(j) \times_{[n] \setminus \{ i,j\} } B(i,j) \right) \times P_j$, 
    \item $A(i,j ) \times_{[n]\setminus\{i,j\}} B$ and $ A(j) \times_{[n]\setminus\{i,j\}} B(i)$ along $ \left( A(i,j) \times_{[n] \setminus \{ i,j\} } B(i) \right) \times P_i$,
    \item $A(i,j ) \times_{[n]\setminus\{i,j\}} B$ and $ A(i) \times_{[n]\setminus\{i,j\}} B(j)$ along $\left( A(i,j) \times_{[n] \setminus \{ i,j\} } B(j) \right) \times P_j$, 
   \end{itemize} 
we hence obtain
\[
  \partial_j \partial_i ( A \ttimes B) \cong ( A \ttimes B) (i,j) \times P_i \times P_j \cong  \partial_i \partial_j ( A \ttimes B)  \, .
\]
Arguing in a similar manner for arbitrary $\omega \subset \{1,\ldots, n\}$ we can work with 
\begin{equation} \label{complicatedproduct} 
   ( A \ttimes B) (\omega) := 2^{| \omega|} \cdot \bigcup_{\omega' \subset \omega}   A( \omega') \times_{\{1, \ldots, n\} \setminus \omega} B(\omega \setminus \omega') 
\end{equation} 
with gluings of components associated to  $\omega', \omega''  \subset \omega$ with $| \omega' \bigtriangleup \omega''| = 1$ (cardinality of symmetric difference)  in order to identify $A \ttimes B$ as a $\mathscr{P}_n$-manifold. 

\end{proof} 

Let $X$ and $Y$ be topological spaces, let $A^d$ and $B^{e}$ be closed oriented $\mathscr{P}_n$-manifolds of dimensions $d$ and $e$ and let $\alpha : A \to X$ and $\beta : B \to Y$ be maps which are compatible with the singularity structures of $A$ and $B$ (see Definition \ref{mapbs}).
Then the induced map $\alpha  \times \beta : A \times B \to X \times Y$ is compatible with our inductive construction of $A \ttimes B$ and we obtain an induced map $\alpha \ttimes \beta : A \ttimes B \to X \times Y$. 
This results in a bilinear map of  bordism theories 
\[
  \ttimes :  \Omega^{\SO, \mathscr{P}_n}_d (X) \times \Omega^{\SO, \mathscr{P}_n}_{e} (Y) \to \Omega^{\SO, \mathscr{P}_n}_{d + e}  ( X \times Y) \, , 
\]
(the theories $\Omega^{\SO, \mathscr{P}_n}_*(-)$ were introduced after Definition \ref{defbordsing}) and this construction extends to relative bordism groups. 
With the natural transformation $u : \Omega^{\SO, \mathscr{P}_n} _*(-)   \to  \HH_*(-) $ from Construction \ref{constr:trafo} we hence obtain the following result. 

\begin{prop} \label{prop:admissible} Let $\times$ denote  the cross product in singular homology. 
Then for  all pairs of topological spaces $(X,S)$ and $(Y, T)$ we have a commutative diagram 
\[
\xymatrix{
        \Omega^{\SO, \mathscr{P}_n}_d (X,S) \times  \Omega^{\SO,\mathscr{P}_n}_{e}(Y, T) \ar[r]^-{\ttimes} \ar[d]^{u \times u} & \Omega^{\SO,  \mathscr{P}_n}_{d+e}(X \times Y, X \times T \cup S \times Y)  \ar[d]^{u} \\ 
               \HH_d(X,S) \times \HH_{e} (Y, T) \ar[r]^-{(a, b) \mapsto 2^n \cdot (a \times b)} & \HH_{d+ e} (X \times Y, X \times T \cup S \times Y) \, . }  
 \]
  \end{prop}

 \begin{rem} The factor $2^n$ appears even if $A$ or $B$ are without singular strata, see Remark \ref{admprodind}  \ref{always}.  
In particular the product on $\Omega^{\SO, \mathscr{P}_n}_* (-)$   is not unital for $n \geq 1$.  
\end{rem}

Now, choose Riemannian metrics $h_i$ on $P_i \in \mathscr{P}$ for $i \geq 1$.
Let $A$ and $B$ be $\mathscr{P}_n$-manifolds and  let $g$ and $h$ be $\mathscr{P}$-compatible metrics on $A$ and $B$ in the sense of Definition \ref{dist}. 
Let $\lambda, \mu >  0$ and $\delta, \epsilon \geq 9$. 
With this choice of $\delta$ and $\epsilon$, the local models $U \times [0,1)^k$ on $A$ and $B$, equipped with the scaled metrics $g_{(\lambda, \delta)}$ and $h_{(\mu, \epsilon)}$ from Construction \ref{constrscal},  can be canonically extended to local models $U \times [0,3)^k$ on which these scaled metrics still restrict to product metrics $g^U \oplus \eta$ and $h^U \oplus \eta$, respectively, with the Euclidean metric $\eta$  on $[0,3)^k$. 

We equip the hexagonal manifold $\mathfrak{X}$ with some admissible Riemannian metric $\sigma$ (see Definition \ref{defadmissible}) with respect to which each side has length $3$.

With these data we construct a metric $g_{(\lambda, \delta)} \tilde{\oplus} h_{(\mu, \epsilon)}$ on $A \ttimes B$ along the inductive construction of $A \ttimes B$ before Definition \ref{admissibleprod}, starting with the product metric $g \oplus h$ on $A \times B$ and working with collar factors $[0,3)^2$ and $[0,3)$ instead of $[0,1)^2$ and $[0,1)$ in Figure \ref{attach}.   
Here it is important that the interchange map on $P_i \times P_i$ is an isometry with respect to $h_i \oplus h_i$.

By the choice of $\delta$ and $\epsilon$ and the metric  $\sigma$ on $\mathfrak{X}$  we hence obtain a $\mathscr{P}$-compatible metric $g_{(\lambda, \delta)} \tilde{\oplus} h_{(\mu, \epsilon)}$ on $A \ttimes B$.

\begin{defn} \label{adm_prod_met} We call  $g_{(\lambda, \delta)} \tilde{\oplus} h_{(\mu ,\epsilon)}$ the {\em admissible product metric} of $g_{(\lambda, \delta)}$ and $h_{(\mu, \epsilon)}$. 
\end{defn} 

We obtain the following version of the well known ``shrinking one factor'' principle. 

\begin{prop} \label{posproduct}  Assume that $A$  and $B$ are compact and $g$ is positive (see Definition \ref{comppos}). 
Then for any $\mu  \geq  1  $ and $\epsilon \geq 9$ there exists $0 < \lambda \leq 1 $ and $\delta \geq 9$ such that for all $\delta' \geq \delta$ the following holds: 
\begin{enumerate}[label={(\roman*)}]
  \item  \label{one} The metric $g_{(\lambda, \delta')} \tilde{\oplus} h_{(\mu, \epsilon)}$  on $A \ttimes B$ is positive.
   \item \label{two} The metric $g_{(\lambda, \delta')}$   is positive.
   \item \label{gentle} Let $C$ be a compact $\mathscr{P}_n$-manifold, let $k$ be a $\mathscr{P}$-compatible metric on $C$, and let $\nu >0 $ and $\theta \geq 9$ be such that the scaled metric $k_{(\nu, \theta)}$ is positive.  
           Then $g_{(\lambda, \delta')} \tilde{\oplus} k_{(\nu, \theta)}$ is positive.
\end{enumerate} 
 
\end{prop}

\begin{proof} Set $\min(  \scal_{\sigma} )  := \min_{x \in \mathfrak{X}} \{  \scal_{\sigma}(x) \}  \in \R$. 
We will use a similar notation for other metrics instead of $\sigma$.
Note that  $\min ( \scal_{\sigma} )  < 0$ by the Gauss-Bonnet formula, since the boundary pieces of $\mathfrak{X}$ are totally geodesic and meet at angles $\pi/2$. 

At each  inductive step in the construction of $A \ttimes B$ we replace two collar factors $[0,3)^2$ (with zero scalar curvature) by a factor $\mathfrak{X}$ equipped with the metric $\sigma$. 
Hence, and more generally  for $\omega \subset \{1, \ldots, n\}$, we obtain with   \eqref{complicatedproduct} 
\[
   \min \big(  \scal_{ \left( g_{(\lambda, \delta)}  \tilde{\oplus}  h_{(\mu, \epsilon)}\right)(\omega) } \big)   \geq  \min_{\omega' \subset \omega}   \{  \min ( \scal_{g_{(\lambda, \delta) }(\omega')} )  + \min(   \scal_{h_{(\mu, \epsilon)} (\omega \setminus \omega')} )  + (n - |\omega|) \cdot \min( \scal_{\sigma})   \} \, .
\]
By Proposition \ref{scalescal} \ref{duo} applied to  
\[
   s := \max_{\omega \subset \{1, \ldots, n\}} \{ | \min ( \scal_{h_{(\mu,\epsilon)}(\omega)})| + n \cdot |\min(\scal_{\sigma}) | \},  
\]
 we find $0 < \lambda \leq 1 $ and $\delta  \geq 9$ with the stated properties. 
\end{proof}

\section{Positive cross products and Toda brackets} \label{Toda} 

Let $X$ and $Y$ be topological spaces and consider the K\"unneth sequence of singular homology groups 
\[
    0 \to \HH_*(X) \otimes \HH_*(Y) \stackrel{\times}{\longrightarrow} \HH_*(X \times Y)  \longrightarrow \Tor(\HH_*(X) \, , \HH_*(Y))_{*-1}  \to 0  \, . 
\]
In this section we study positive homology classes (see Definition \ref{poshom}) related to the homological cross product $\times$ and the Tor-term in this sequence.

\begin{setting} \label{constrtoda} Let  
\[
   [\alpha : A \to X] \in \Omega _d^{SO, \mathscr{Q}_n}(X) \text{ and }  [\beta : B \to Y] \in \Omega _{e}^{SO, \mathscr{Q}_n}(Y) 
\]
with $\mathscr{Q}_n$-manifolds $A$ and $B$  \rot{and let $a \in \HH_d(X)$ and $b \in \HH_{e}(Y)$} be the images of these bordism classes under the natural transformation $u$ from Construction \ref{constr:trafo}. 
\end{setting} 

Propositions  \ref{prop:admissible} and  \ref{posproduct} \ref{one}  imply the following result. 

\begin{prop} \label{prop:prodpos} Assume that at least one of the  Baas-Sullivan manifolds $A$ or $B$ is equipped with a $\mathscr{Q}$-compatible positive metric (see Definition \ref{comppos}). 
Then the class $2^n  \cdot \big( a \times b \big)  \in \HH_{d+ e} (X \times Y)$ is positive. 
\end{prop} 
 
Next we discuss the Tor term in the K\"unneth sequence. 
Let $r \geq 2$ be an integer with $r a= 0 = rb$.
Let  $(C_*(X), \partial)$ and $(C_*(Y), \partial)$ be the integral chain complexes of $X$ and $Y$. 
We  pick chains $\overline{a} \in C_{d+1}(X)$ and $\overline{b}  \in C_{e+1}(Y)$ whose boundaries represent $ra$ and $rb$ respectively. 
The cycle  
\begin{equation} \label{defToda} 
    \frac{1}{r} \cdot  \partial ( \overline a \otimes \overline b ) \in (C_*(X) \otimes C_*(Y))_{d + e + 1}   
\end{equation} 
represents a {\em Toda bracket} coset
\[
     \langle a , r , b \rangle  \subset \HH_{d+e+1} (X \times Y) 
\]
with respect to the submodule  $ (a \times \HH_{e+1}(Y) ) \oplus ( \HH_{d+1}(X) \times b) \subset \HH_{d + e +1} (X \times Y)$, which is independent from the choice of $\overline{a}$ and $\overline{b}$. 
It is well known \cite{EML}*{Section  12} that such Toda brackets generate a submodule of $\HH_{d+e+1} (X \times Y)$ which maps surjectively onto $\Tor(\HH_*(X), \HH_*(Y))_{d+e}$.  
In the following we  give a bordism theoretic description of Toda brackets. 

By Proposition \ref{isobordhom} there exists $m \geq 0$ such that 
\[
   2^m\cdot r \cdot  [\alpha: A \to X] = 0  \text{ and }   2^m \cdot r \cdot [\beta : B \to Y] = 0 \, . 
\]
Hence, possibly after passing to some larger $n$, there are  compact oriented $\mathscr{Q}_n$-manifolds $V$ and $W$ with boundaries $\partial_0 V = \coprod_{2^m \cdot r} A$ and $\partial_0 W = \coprod_{2^m \cdot r} B$ such that $\coprod_{2^m \cdot r}  (A \stackrel{\alpha}{\to} X)$ and $\coprod_{2^m \cdot r} (B \stackrel{\beta}{\to} Y)$ can be extended to  maps $\overline{\alpha} : V \to X$ and $\overline{\beta} : W \to Y$ where $\overline{\alpha}$ and $\overline{\beta}$ are compatible with the singularity structures of $V$ and $W$.

By \eqref{defToda} and Proposition  \ref{prop:admissible} the coset   $ 2^{m+n}  \cdot \langle a, r, b \rangle \subset \HH_*(X \times Y)$ is  represented by 
\begin{equation} \label{Toda_repr} 
  (  \overline{\alpha} \ttimes \beta ) \cup ( \alpha \ttimes   \overline{\beta}) :  ( V  \ttimes   B)   \cup_{\partial_0  V \ttimes  B  = A \ttimes \partial_0 W}  ( A \ttimes  W)  \to X \times Y \, . 
\end{equation}

Let  $A$ and $B$ be equipped with $\mathscr{Q}$-compatible positive metrics $g$ and $h$. 
The metrics $\coprod_{2^m \cdot r}  g$ on $\coprod_{2^m \cdot r} A$ and  $\coprod_{2^m \cdot r} h$ on  $\coprod_{2^m \cdot r} B$,  can be extended to (not necessarily positive) $\mathscr{Q}$-compatible metrics $\overline g$ and $\overline h $ on $V$ and $W$ (compare the proof of Lemma \ref{compat}). 

By Proposition  \ref{scalescal} \ref{uno}  we find $\delta_0 , \epsilon_0 \geq 9$ such that for all $\delta \geq \delta_0$ and $\epsilon \geq \epsilon_0$ the scaled metrics $g_{(1, \delta)}$ and $h_{(1, \epsilon)}$ are positive.

Choose $(\lambda, \delta)$ for $A$ according to Proposition \ref{posproduct} for the scaled metric ${\overline h}_{(1, \epsilon_0)}$ on $W$, and in an analogous fashion choose  $(\mu, \epsilon)$ for $B$ for the scaled metric ${\overline g}_{(1, \delta_0)}$ on $V$. 
With these choices the  admissible product metrics  ${\overline g}_{(1, \delta_0)} \tilde{\oplus} h_{(\mu, \epsilon)}$ on $V \ttimes B$  and $g_{(\lambda, \delta)} \tilde{\oplus} {\overline h}_{(1, \epsilon_0)}$ on $A \ttimes W$ are positive by Proposition  \ref{posproduct} \ref{one}. 
In order to glue the induced metrics on the common boundary $\coprod_{2^m \cdot r} A \ttimes B$ we need the following result.

\begin{lem} \label{isotopydifficult} The metrics $g_{(1,\delta_0)} \tilde{\oplus} h_{(\mu, \epsilon)}$ and $g_{(\lambda, \delta)} \tilde{\oplus} h_{(1,\epsilon_0)}$ on  $A  \ttimes B$ are isotopic,  and hence concordant, through positive $\mathscr{Q}$-compatible metrics. 
\end{lem} 

\begin{proof} Set $\Lambda = [\lambda,1] \subset \R$ and choose $\delta'_0 \geq \delta_0, \delta$ according to  Proposition \ref{scalescal} \ref{uno}  for this $\Lambda$. 
 We find isotopies through positive $\mathscr{Q}$-compatible metrics on $A$: 
  \begin{itemize} 
      \item from $g_{(1, \delta_0)}$ to $g_{(1, \delta'_0)}$, by the choice of $\delta_0$; 
      \item from $g_{(1, \delta'_0)}$ to $g_{(\lambda, \delta'_0)}$, by the choice of $\delta'_0$;
      \item from $g_{(  \lambda, \delta'_0 )}$ to $g_{( \lambda, \delta)}$, by the choice of $(\lambda, \delta)$ and by Proposition \ref{posproduct} \ref{two}. 
  \end{itemize}
Hence, by the choice of $(\mu, \epsilon)$,  we obtain a smooth isotopy from $g_{( 1 , \delta_0)} \tilde{\oplus} h_{(\mu, \epsilon)}$ to $g_{( \lambda, \delta)} \tilde{\oplus} h_{( \mu , \epsilon)}$ through positive $\mathscr{Q}$-compatible metrics, see Proposition \ref{posproduct} \ref{gentle}. 

In an analogous fashion we find a smooth isotopy from $g_{(\lambda, \delta)} \tilde{\oplus} h_{(1  , \epsilon_0)}$ to $g_{( \lambda, \delta)} \tilde{\oplus} h_{( \mu , \epsilon)}$ through positive $\mathscr{Q}$-compatible metrics, thus finishing the proof of Lemma \ref{isotopydifficult}. 
\end{proof} 

We obtain the following counterpart of Proposition \ref{prop:prodpos}. 

\begin{prop} \label{tor_pos} \rot{We work in Setting \ref{constrtoda}}   and assume in addition that  both of the Baas-Sullivan manifolds $A$ and $B$ are equipped with $\mathscr{Q}$-compatible positive metrics. 
Let $r \geq 2$ be such that $ra = 0 = rb$. 
Then for each element $x \in  \langle a, r, b \rangle  \subset H_{d + e +1} ( X \times Y)$ there exists $\ell  \geq 0$ such that $2^{\ell}  \cdot x$ is positive. 
\end{prop} 

\begin{proof} Using the notation introduced after Proposition \ref{prop:prodpos} the $\mathscr{Q}_n$-manifold  
\[
  (V  \ttimes  B)  \cup_{\partial_0  V \ttimes  B  = A \ttimes \partial_0 W} ( A  \ttimes  W )
\]
 in \eqref{Toda_repr}  carries a $\mathscr{Q}$-compatible positive metric by Lemma \ref{isotopydifficult}.
Hence   the class $x' \in 2^{m+ n}   \cdot  \langle a, r, b\rangle$ represented by $ (V  \ttimes  B)  \cup ( A  \ttimes  W )  \to X \times Y$ is positive.

\rot{It is enough to show Proposition \ref{tor_pos} for $x \in 2^{m+n} \cdot \langle a ,r ,b \rangle$.} 
Given such $x$ we have $x - x' \in  \big( a \times \HH_{e+1}(Y) \big) \oplus \big( \HH_{d+1}(X) \times b \big) \subset \HH_*(X \times Y)$, and 
by Propositions \ref{isobordhom} and \ref{prop:prodpos} and  since $a$ and $b$ are positive there exists $\ell \geq 0$ such that $2^{\ell} \cdot (x - x')$ is positive. 
Using that $x'$ is positive we conclude that $2^{\ell} \cdot x$ is positive. 
\end{proof} 

Let $\Gamma_1$ and $\Gamma_2$ be finite groups of odd order and set $X = B \Gamma_1$ and $Y = B \Gamma_2$. 
Then  each $x  \in \tilde \HH_*(X \times Y)$  is of odd order and hence for all $m_0 \geq 0$ there exists $m \geq m_0$ with $2^m \cdot x = x$. 
By  Corollary  \ref{iso_odd_order} and Propositions \ref{prop:prodpos} and \ref{tor_pos} we conclude: 

\begin{cor} \label{summary} Let $a \in H_{d}(B \Gamma_1)$ and $b \in \HH_e(B \Gamma_2)$ where $d, e \geq 0$. 
\begin{enumerate}[label={(\roman*)}]
    \item \label{oans} If either $a$ or $b$ is positive, then $a \times b$ is positive. 
    \item \label{zwoa} Let $r \geq 2$ with $ra = 0 = rb$ and let $a$ and $b$ be positive.
           Then $\langle a, r, b\rangle \subset \HH_{d+e+1}(B  \Gamma_1 \times B\Gamma_2 )$ only contains positive classes. 
\end{enumerate} 
\end{cor} 
This result will be crucial for the computations in the next sections.

 \begin{rem} \rot{One can show that the product $\ttimes$ on (relative) bordism groups $\Omega^{\SO, \mathscr{P}_n}_*$ considered in Proposition \ref{prop:admissible} is graded commutative and associative.
Corollary \ref{summary}, which is sufficient for the remainder of our paper, does not depend on these facts.} 
\end{rem}

\section{Homology of abelian groups}      \label{hom_p_1} 

Let $p$ be an odd prime.  
Given an integer $\alpha \geq 1$ we denote by $G_\alpha$ the cyclic group of order $p^{\alpha}$  with generator $g_\alpha$ and neutral element  $1_\alpha$. 
The group operation in $G_\alpha$ is written multiplicatively. 
We denote by $\Z G_{\alpha}$ the integral group ring of $G_\alpha$. 

Let $ (  C(\alpha)_*, \partial_*) $ be  the $\Z$-graded $\Z$-free chain complex with one generator $c_d$ in each degree $d \geq 0$ and differential 
\[ 
     \partial( c_d) = \begin{cases} p^{\alpha} \cdot c_{d-1} & \text{ for even } d \geq 2\, ,  \\
                                                                0 & \textrm{ for odd } d  \text{ and for } d = 0\, . 
                                                                 
                              \end{cases} 
\]
This is the cellular chain complex with integer coefficients of the standard CW-model of the classifying space $B G_{\alpha}$  with one cell in each non-negative dimension. 
We hence recover the well known computation (see \cite{Brown}*{(II.3.1)})  
\[              
   \HH_{d} (C(\alpha)_*, \partial_*) \cong \HH_{d} ( B G_{\alpha} ) = \begin{cases} \Z = \langle [c_0] \rangle  &  \text{ for } d = 0 \, , \\
                       \Z/p^\alpha = \langle  [c_d] \rangle  & \text{ for  odd } d  \, , \\
                  0 & \text{ for even } d \geq 2\, .  \end{cases}
 \]
For $n \geq 1$ and $1 \leq \alpha_1 \leq \cdots \leq \alpha_n$ we consider the abelian $p$-group 
\[
   \Gamma = G_{\alpha_1}  \times \cdots \times G_{\alpha_n}  
 \]
and obtain
\[
   \HH_*( B\Gamma ) \cong \HH_*(C^{(1)}_* \otimes \cdots \otimes C^{(n)}_* ) , 
\]
where $C^{(i)}_* = C(\alpha_i)_*$, $i = 1, \ldots, n$, refers to the $i$th cyclic factor in the group $\Gamma$. 

Sometimes we will work with  the reduced chain complex $\tilde C(\alpha)_* := C(\alpha)_* / \langle c_0 \rangle$ of $C(\alpha)_*$. 
Note the canonical direct sum decomposition $C(\alpha)_* = \tilde C(\alpha)_* \oplus \langle c_0 \rangle$ of chain complexes and the isomorphism $\tilde \HH_*(\widehat{B\Gamma}) \cong \HH_*(\tilde C(\alpha_1)_* \otimes \cdots \otimes \tilde C(\alpha_n)_*)$, where $\widehat{B\Gamma} = BG_{\alpha_1} \wedge \cdots \wedge BG_{\alpha_n}$ is the smash product of pointed classifying spaces.
In general we obtain a direct sum decomposition of chain complexes
\begin{equation}  \label{dirsum} 
    C(\alpha_1)_* \otimes \cdots \otimes C(\alpha_n)_* =  \bigoplus_{\substack{0 \leq k \leq n \\ 1 \leq i_1< \cdots <  i_k \leq n }}       \tilde C(\alpha_{i_1})_*  \otimes \cdots \otimes \tilde C(\alpha_{i_k})_* \, ,  
\end{equation} 
where the summand for $k = 0$ is equal to $\langle c_0 \otimes \cdots \otimes c_0 \rangle \subset C(\alpha_1)_* \otimes \cdots \otimes C(\alpha_n)_*$, by definition. 
For analyzing $\HH_*( C(\alpha_1)_* \otimes \cdots \otimes C(\alpha_n)_*)$ it is hence important to provide convenient generators of $\HH_*(  \tilde C(\alpha_{i_1})_*  \otimes \cdots \otimes \tilde C(\alpha_{i_k})_* )$, for $1 \leq k \leq n$ and $1 \leq i_1 < \ldots < i_k \leq n$.  
In Proposition \ref{generated} we will do this for $\HH_*(\tilde C(\alpha_1)_* \otimes \cdots \otimes \tilde C(\alpha_n)_* )$, the other cases are analogous. 
We first we write down cycle representatives of iterated Toda brackets.

\begin{constr} \label{deftor} Let $k \geq 1$, let $1 \leq \beta_1 \leq \cdots \leq \beta_k$ and let $m_1, \ldots, m_k$ be positive integers. 
We define a cycle in $\tilde C(\beta_1)_* \otimes \cdots \otimes \tilde C(\beta_k)_*$  of degree $2m_{1} + \cdots + 2m_{k} - 1$ by 
\begin{equation*}
\label{tau} 
   \mathscr{T}   (c^{(1)}_{2m_{1} -1}  , \ldots, c^{(k)}_{2m_{k}-1}) : = \frac{1}{p^{\beta_{1}}}  \partial ( c^{(1)}_{2m_{1}} \otimes \cdots \otimes c^{(k)}_{2m_{k} }) = \sum_{\ell=1}^k p^{\beta_{\ell} - \beta_{1}} \cdot ( c^{(1)}_{2m_{1} } \otimes \cdots \otimes c^{(\ell)}_{2m_{\ell}-1} \otimes \cdots \otimes c^{(k)}_{2m_{k}} ) \,  .  
\end{equation*} 
Clearly the corresponding homology class satisfies $p^{\beta_{1}} \cdot  [ \mathscr{T}  (c^{(1)}_{2m_{1} -1}  , \ldots, c^{(k)}_{2m_{k}-1})]  = 0$. 
For $k=1$ we have  $\mathscr{T}  (c^{(1)}_{2m_{1}-1}) = c^{(1)}_{2m_{1}-1}$, and for $k \geq 2$ we obtain {\em iterated Toda brackets}. 
More precisely, setting $h_i := [c^{(i)}_{2m_i-1}] \in \HH_{2m_{i}-1}(\tilde C(\beta_i)_*) $ for $i = 1, \ldots, k$,  we obtain
\[
   [ \mathscr{T}  (c^{(1)}_{2m_{1}-1}, \ldots, c^{(k)}_{2m_{k}-1})] \in   \langle  h_1  , p^{\beta_1} ,   p^{\beta_2-\beta_1} \langle h_2 , p^{\beta_2}, \cdots \langle h_{k-1} , p^{\beta_{k-1}} , p^{\beta_k - \beta_{k-1}} h_k \rangle \cdots \rangle \rangle \, . 
\]
\end{constr} 

We can now construct specific generators of $\HH_*(\tilde C(\alpha_1)_* \otimes \cdots \otimes \tilde C(\alpha_n)_* )$. 
Let $1 \leq j \leq n$, let $1 \leq i_1 < \cdots < i_j \leq n$ and let $m_{1}, \ldots, m_{j}$ be positive integers. 
Let $(s_1, \ldots, s_{n-j})$ with $1 \leq s_1 < \cdots < s_{n-j} \leq n$ be the unique family complementary to $(i_1, \ldots, i_j)$ (this family is empty for $j = n$) and let $d_1, \ldots, d_{n-j}$ be further positive integers.  
Suppressing  a signed permutation of tensor factors we obtain a cycle
\[
   \mathscr{T}  (c_{2m_{1}-1}^{(i_1)} , \ldots, c_{2m_{j}-1}^{(i_j)}) \otimes c_{2d_1-1}^{(s_1)} \otimes \cdots \otimes c_{2d_{n-j}-1}^{(s_{n-j})} \in \tilde C_*^{(1)} \otimes \cdots \otimes \tilde C_*^{(n)} \subset C(\alpha_1)_* \otimes \cdots \otimes C(\alpha_n)_*\, . 
\]
In the following we will call cycles of this sort  {\em special}. 

\begin{prop} \label{generated} $ \HH_*(\tilde C(\alpha_1)_* \otimes \cdots \otimes \tilde C(\alpha_n)_* )$ is generated by special cycles with $i_1 = 1$.
\end{prop}
\begin{proof} We apply induction on $n$. 
In the induction step we set $\mathscr{C}_*^n  :=  \tilde C_*^{(1)} \otimes \cdots \otimes \tilde C_*^{(n)}$ and consider the exact K\"unneth sequence 
\[ 
   0 \longrightarrow  \HH_*(\mathscr{C}_*^n ) \otimes \HH_*(\tilde C_*^{(n+1)})  \longrightarrow \HH_*(\mathscr{C}_*^n \otimes \tilde C_*^{(n+1)} )  \longrightarrow \Tor( \HH_*(\mathscr{C}^n_*) , \HH_*(\tilde C_*^{(n+1)}))_{* -1}  \longrightarrow 0 . 
\]
By the induction hypothesis, the construction of  $ \Tor( \HH_*(\mathscr{C}_*^n) , \HH_*(\tilde C_*^{(n+1)}))$ and the assumption that  $\alpha_1  \leq \alpha_{n+1}$, Toda brackets of the form 
\[
   \langle    [ \mathscr{T}  (c^{(1)}, \ldots, c^{(i_j)}) \otimes c^{(s_1)} \otimes \cdots \otimes c^{(s_{n -j})} ] , p^{\alpha_1}, p^{\alpha_{n+1}  - \alpha_1} [c^{(n+1)} ] \rangle 
\]
map to a generating set of $ \Tor( \HH_*(\mathscr{C}_*^n) , \HH_*(\tilde C_*^{(n+1)}))$.

This Toda bracket contains $ [ \mathscr{T}  (c^{(1)}, \ldots, c^{(i_j)}, c^{(n+1)}) \otimes c^{(s_1)} \otimes \cdots \otimes c^{(s_{n -j})} ] $ (up to sign), and hence special cycles  $  \mathscr{T}  (c^{(1)}, \ldots, c^{(i_j)}, c^{(n+1)}) \otimes c^{(s_1)} \otimes \cdots \otimes c^{(s_{n -j})} $ map to a generating set of  $ \Tor( \HH_*(\mathscr{C}_*^n) , \HH_*(\tilde C_*^{(n+1)}))$. 

The image of the left-hand map in the K\"unneth sequence satisfies the claim by the induction assumption. 
\end{proof}   
 
\begin{ex} Let $n = 3$ and $\alpha_1 = 1$, $\alpha_2 = 2$ and $\alpha_3 = 3$. 
Then 
\[
  0 \neq [ \mathscr{T}  (c^{(2)}_{1}, c^{(3)}_{1}) \otimes c^{(1)}_{1 } ] \in \HH_4(B\Gamma) \, .
\]
Proposition \ref{generated} can be illustrated in this case by computing $  \mathscr{T}   (c^{(2)}_{1}, c^{(3)}_{1}) \otimes c^{(1)}_{1 }$ as 
\[
   - \mathscr{T}   (c^{(1)}_{1},  c^{(3)}_{1}) \otimes c^{(2)}_{1 }
     -  p \cdot   \mathscr{T}   (c^{(1)}_{1}, c^{(2)}_{1}) \otimes c^{(3)}_{1 }  = - \mathscr{T}  (c^{(1)}_{1},  c^{(3)}_{1}) \otimes c^{(2)}_{1 }
           - \partial ( c^{(1)}_{2} \otimes c^{(2)}_{2} \otimes c^{(3)}_{1} ) \, . 
\] 
\end{ex} 

Next we will derive some explicit formulas for maps in group homology induced by group homomorphisms. 
We consider the homological chain complex in non-negative degrees
\[
    (F(\alpha)_*, \partial_*) :=  \big( \cdots  \longrightarrow \Z G_\alpha   \stackrel{ \nu_\alpha}{\longrightarrow}        \Z G_\alpha  \stackrel{\tau_\alpha}{\longrightarrow}      \Z G_\alpha \stackrel{ \nu_\alpha}{\longrightarrow}  \Z G_\alpha  \stackrel{\tau_\alpha}{\longrightarrow}  \Z G_{\alpha}   \big) 
\]
where the differentials are given by multiplication with $\tau_\alpha:= g_\alpha -1_{\alpha}$ and  $\nu_\alpha:= \sum_{i=0}^{p^{\alpha} -1} (g_\alpha)^i $, respectively. 
With the augmentation map $\varepsilon_{\alpha}: \Z G_{\alpha} \to \Z$ induced by the group homomorphism $G_{\alpha} \to \{1\}$  we obtain an exact sequence
\[
   \cdots  \longrightarrow \Z G_\alpha   \stackrel{\nu_\alpha}{\longrightarrow}        \Z G_\alpha  \stackrel{\tau_\alpha}{\longrightarrow}      \Z G_\alpha \stackrel{ \nu_\alpha}{\longrightarrow}  \Z G_\alpha  \stackrel{\tau_\alpha}{\longrightarrow}  \Z G_{\alpha}   \stackrel{\varepsilon_\alpha}{\longrightarrow} \Z \to 0 \, . 
\]
In other words $(F(\alpha)_*, \partial_*)$ is a $\Z G_{\alpha}$-free resolution of the $\Z G_{\alpha}$-module $\Z$, see \cite{Brown}*{(I.6.3)}. 
Note the canonical isomorphism of chain complexes $C(\alpha)_*  =    F(\alpha)_* \otimes_{\Z G_\alpha} \Z$. 

Let $\alpha, \beta, \lambda \in \N_{>0}$ with  $p^{\beta} \mid  \lambda \cdot p^{\alpha}$, and consider the group homomorphism
\[
   \phi: G_{\alpha}   \to G_{\beta}   \, , \quad g_{\alpha}    \mapsto ( g_{\beta})^\lambda \, . 
\]
Then each $\Z G_{\beta}$-module can be regarded as a $\Z G_{\alpha}$-module via  the ring map 
\[
   \Z \phi : \Z G_{\alpha} \to \Z G_{\beta} .
\] 
With this convention the assignments (using $\lambda \cdot p^{\alpha - \beta} \in \N_{\rot{>0}}$) 
\begin{eqnarray*}
    \phi_{2m}(1_\alpha) & := & (  \lambda \cdot p^{\alpha - \beta})^m \cdot 1_\beta \, , \\
    \phi_{2m-1}(1_\alpha) & := & ( \lambda \cdot p^{\alpha - \beta})^{m-1} \cdot \sum_{i = 0}^{\lambda-1} (g_{\beta})^i 
\end{eqnarray*}
uniquely extend  to  $\Z G_{\alpha}$-linear maps $\Z G_{\alpha} \to \Z G_{\beta}$ and an explicit  computation 
\begin{align*} \phi_{2m-1} (\nu_\alpha \cdot 1_{\alpha}) & =  (\lambda p^{\alpha-\beta})^{m-1} \cdot \sum_{i=0}^{p^{\alpha}-1} (g_{\beta})^{i \lambda} \cdot \sum_{j=0}^{\lambda-1} (g_{\beta})^{j}  \\
                                                                                        & = (\lambda p^{\alpha - \beta})^m \cdot \sum_{i=0}^{p^{\beta}-1} (g_{\beta})^i =   \nu_{\beta} \cdot \phi_{2m} (1_{\alpha}), 
\end{align*} 
and similar to obtain $\phi_{2m} (\tau_\alpha \cdot 1_\alpha) = \tau_\beta \cdot \phi_{2m+1} (1_\alpha)$,  shows  that we obtain an augmentation preserving  map of $\Z G_{\alpha}$-linear chain complexes 
\[
    \xymatrix{
\cdots \ar[r] &   \Z G_\alpha   \ar[d]^{\phi_4}  \ar[r]^{\nu_\alpha}    &      \Z G_\alpha \ar[d]^{\phi_3} \ar[r]^{\tau_\alpha}   &          \Z G_\alpha \ar[d]^{\phi_2}  \ar[r]^{\nu_\alpha }    &  \Z G_\alpha \ar[d]^{\phi_1 } \ar[r]^{\tau_\alpha}    & \Z G_\alpha \ar[d]^{\phi_0 = \Z \phi }   \\
\cdots \ar[r] &     \Z G_\beta                    \ar[r]^{\nu_\beta}     &     \Z G_\beta                       \ar[r]^{\tau_\beta}   &           \Z G_\beta                       \ar[r]^{\nu_\beta}    &  \Z G_\beta                          \ar[r]^{\tau_\beta}    &  \Z G_\beta                  
      } 
\]
After applying the functor $- \otimes_{\Z G_\alpha} \Z$  we obtain the following result. 

\begin{prop} \label{compute} 
The induced chain map $\phi_*: C(\alpha)_* \to C(\beta)_*$ is given by
\begin{eqnarray*} \label{homeins} 
   \phi_{2m}(c_{2m})  =    & (\lambda \cdot p^{\alpha-\beta})^m \cdot c_{2m}   &  \text{ for } m \geq 0  \, , \\
   \phi_{2m-1}(c_{2m-1})  = &    \lambda \cdot (\lambda \cdot p^{\alpha- \beta})^{m-1} \cdot c_{2m-1} &   \text{ for } m \geq 1\,  . 
\end{eqnarray*} 
\end{prop} 

Note that the map induced in homology  by $\phi_*$ can be identified with the map 
\[
   (B \phi)_* :   \HH_*(B G_{\alpha}) \to  \HH_*(BG_{\beta}) \, ,
\]
compare \cite{Brown}*{(II.6.1)}. 

\begin{lem} \label{diag} 
Consider the diagonal map $ \Delta : G_{\alpha}   \to G_{\alpha} \times G_{\alpha} $, $g \mapsto (g,g)$. 
Then the induced map in homology $\Delta_* : \HH_*( C(\alpha)_* ) \to \HH_*(C(\alpha)_* \otimes C(\alpha)_*)$ satisfies
\[
    \Delta_*( [c_{2m + 1}] ) =    \left[  \sum_{i=0}^{2m+1}  c_i \otimes c_{2m+1-i}  \right]  \, . 
\]
\end{lem} 

\begin{proof} Obviously $\sum_{i=0}^{2m+1}  c_i \otimes c_{2m+1-i}$ is a cycle in $C(\alpha)_* \otimes C(\alpha)_*$. 
It is enough to show Lemma \ref{diag} after passing to coefficients $\Z/p^{\alpha}$. 
Using the K\"unneth isomorphism  $  \HH_* ( B G_\alpha \times  B G_\alpha; \Z/p^{\alpha}) \cong  \HH_* ( B G_\alpha ; \Z/p^{\alpha}) \otimes \HH_* (B G_{\alpha}; \Z/p^{\alpha})$  the claim now follows from the well known ring structure of $ \HH^*(B G_\alpha ; \Z/p^\alpha)$. 
\end{proof} 

\begin{defn} A cycle $c \in C(\alpha_1)_* \otimes \cdots \otimes C(\alpha_n)_*$ is called {\em positive},  if the homology class $[c] \in   \HH_*(B\Gamma)$  is positive with respect to $\mathscr{Q}$ in the sense of Definition \ref{poshom}.
\end{defn} 

Obviously the cycles $c_{2m-1} \in C(\alpha)_{2m-1}$ are positive for $m \geq 2$ since these can be represented by classifying maps of lens spaces $S^{2m-1}/ ( \Z/p^{\alpha}) \to B G_\alpha$. 
Furthermore the tensor product of two cycles, one of which is positive, is itself positive by Corollary \ref{summary} \ref{oans}, and for $m_1 , m_2 \geq 2$ the cycle $\mathscr{T} (c_{2m_1-1}, c_{2m_2-1}) \in C(\alpha_1)_* \otimes C(\alpha_2)_*$ is positive by Corollary \ref{summary} \ref{zwoa}.  
We will now identify some more positive cycles in $C(\alpha_1)_* \otimes C(\alpha_2)_*$.

\begin{prop} \label{calculate}  For $m \geq 2$ the following cycles in $C(\alpha_1)_* \otimes C(\alpha_2)_*$ are positive:  
   \begin{enumerate}[label={(\roman*)}] 
    \item  \label{tvo} $ p \cdot \mathscr{T}(c_1, c_{2m-1})$ and $p \cdot  \mathscr{T}(c_{2m-1}, c_1)$,  
    \item  \label{tre} $\mathscr{T}(c_1, c_{2m-1})$ and $\mathscr{T} (c_{2m-1}, c_1)$, if $\alpha_1 < \alpha_2$.
   \end{enumerate}  
\end{prop} 

\begin{proof} Given  $\gamma \in \N_{>0} $ we consider the group homomorphism $\phi_{\gamma}: G_{\alpha_1} \to G_{\alpha_2}$ defined by  $g_{\alpha_1} \mapsto (g_{\alpha_2})^{ \gamma \cdot ( p^{\alpha_2 - \alpha_1})}$ and the resulting homomorphism
\[
   f_{\gamma} :  G_{\alpha_1} \stackrel{\Delta}{\lra}  G_{\alpha_1} \times G_{\alpha_1} \stackrel{\id \otimes \phi_{\gamma}}{\lra}  G_{\alpha_1} \times G_{\alpha_2} \, . 
\]
We claim that for $m \geq 1$ the image of $(f_{\gamma})_* ( [c_{2m+1}])$ in $  \HH_*(\tilde C(\alpha_1)_* \otimes \tilde C(\alpha_2)_*)$ is equal to 
\begin{equation} 
     \gamma \cdot [ \mathscr{T}  ( c_{2m-1}, c_1)]  + \cdots + \gamma^{m} \cdot [ \mathscr{T} (c_1, c_{2m-1}) ]  \, . 
\label{coprod} 
 \end{equation} 
This follows from Lemma \ref{diag} and Proposition \ref{compute} (with $\alpha = \alpha_1$, $\beta = \alpha_2$, $\lambda = \gamma  \cdot p^{\alpha_2 - \alpha_1}$), which gives us 
\begin{align*} 
  (\id \times \phi_{\gamma})_* \left(  \sum_{i=1}^{2m} c_i \otimes c_{2m+1-i}\right) & = \sum_{j=1}^{m}  \big( c_{2m-2j+1} \otimes ( \gamma^j c_{2j})  + c_{2m-2j+2} \otimes  (p^{\alpha_2 - \alpha_1} \gamma^j  c_{2j-1} ) \big) \\ 
        & = \sum_{j=1}^{m}  \gamma^j \cdot \big( c_{2m-2j+1} \otimes c_{2j}  + c_{2m-2j+2} \otimes  p^{\alpha_2 - \alpha_1} c_{2j-1} \big)\\
        & = \sum_{j=1}^m \gamma^j \cdot \mathscr{T} (c_{2m-2j+1}, c_{2j-1}) \, ,
\end{align*}                
where the last equation uses Definition \ref{deftor}.

For $m \geq 2$ we have   $p \mid (p^{m} - p)$, but $p^2 \nmid (p^{m} -p)$.
Together with \eqref{coprod} and the fact that $c_{2m+1}$ and $\mathscr{T} (c_{2s-1}, c_{2t-1})$  for $s,t  \geq 2$ are positive, this implies that suitable linear combinations of $(f_1)_*([c_{2m+1}])$ and $(f_{p})_*([c_{2m+1}])$, which define positive classes  in $\HH_{2m+1}( C(\alpha_1)_* \otimes C(\alpha_2)_*)$, map to $p \cdot  [\mathscr{T}  (c_{2m-1}, c_1)]$ and to $p \cdot [ \mathscr{T}  (c_1, c_{2m-1})]$ in $ \HH_*(\tilde C(\alpha_1)_* \otimes \tilde C(\alpha_2)_*)$. 

Since all cycles in $C(\alpha_1)_{2m+1} \otimes C(\alpha_2)_0$ and  $C(\alpha_1)_0 \otimes C(\alpha_2)_{2m+1}$ are positive  this  finishes the proof of part \ref{tvo}. 
 
For part \ref{tre} let $\alpha_1 < \alpha_2$, consider the group homomorphism $\phi : G_{\alpha_2} \to G_{\alpha_1}$  defined by $g_{\alpha_2}   \mapsto  g_{\alpha_1}$ and the resulting homomorphism 
\[
   f  : G_{\alpha_2} \stackrel{\Delta}{\lra}   G_{\alpha_2} \times G_{\alpha_2} \stackrel{\phi \times \id}{\lra}  G_{\alpha_1}  \times G_{\alpha_2}\, . 
\]
We claim that for $m \geq 1$  the image of $f_* ( [c_{2m+1}])$ in $ \HH_*(\tilde C(\alpha_1)_* \otimes \tilde C(\alpha_2)_*)$ is equal to  
\begin{equation} \label{smash} 
       [\mathscr{T}   ( c_1, c_{2m-1})] + \ldots + p^{(m-1) (\alpha_2-\alpha_1)} \cdot [ \mathscr{T}   (c_{2m-1}, c_1)] \, . 
 \end{equation} 
We argue similarly as before, observing that by Proposition \ref{compute} (with $\lambda = 1$) we get 
\begin{align*} 
   (\phi \times \id)_* \left(  \sum_{i=1}^{2m} c_i \otimes c_{2m+1-i}\right) & = \sum_{j=1}^{m}  \big(  p^{(j-1)( \alpha_2 - \alpha_1)}  c_{2j-1}  \otimes c_{2m-2j+2}  +  p^{j ( \alpha_2 - \alpha_1) } c_{2j}  \otimes   c_{2m - 2j+1}  \big) \\ 
    & = \sum_{j=1}^{m} p^{(j-1)( \alpha_2 - \alpha_1)}  \cdot \big( c_{2j-1} \otimes c_{2m-2j+2}  + c_{2j} \otimes  ( p^{\alpha_2 - \alpha_1} c_{2m-2j+1} )\big)\\
        & = \sum_{j=1}^m p^{(j-1)(\alpha_2- \alpha_1)} \cdot \mathscr{T}(c_{2j-1}, c_{2m-2j+1}) \, .
\end{align*}               
Equation \eqref{smash} together with part \ref{tvo}  implies that for $m \geq 2$ there is a positive cycle in $C(\alpha_1)_* \otimes C(\alpha_2)_*$ that maps to $[\mathscr{T}  (c_1, c_{2m-1})]  \in \HH_*( \tilde C(\alpha_1)_* \otimes \tilde C(\alpha_2)_*)$. 
Similar as before this implies that the cycle $\mathscr{T}  (c_1, c_{2m-1}) \in C(\alpha_1)_* \otimes C(\alpha_2)_*$ is positive. 

By \eqref{coprod} applied to $\gamma = 1$ and the positivity of $c_{2m+1}$ and $\mathscr{T} (c_{2s-1}, c_{2t-1})$ for  $s,t  \geq 2$  there is a positive cycle in $C(\alpha_1)_* \otimes C(\alpha_2)_*$ that maps to 
\[
    [\mathscr{T} (c_1, c_{2m-1})] + [\mathscr{T}(c_{2m-1}, c_1)] \in \HH_*( \tilde C(\alpha_1)_* \otimes \tilde C(\alpha_2)_*)
\]
and hence (since $\mathscr{T}  (c_1, c_{2m-1})$ has already been verified as positive) a positive cycle that maps to  $[\mathscr{T}(c_{2m-1}, c_1)] \in   \HH_*( \tilde C(\alpha_1)_* \otimes \tilde C(\alpha_2)_*)$. 
Similar as before this implies that  $\mathscr{T} (c_{2m-1}, c_1) \in C(\alpha_1)_* \otimes C(\alpha_2)_*$ is positive, finishing the proof of part \ref{tre}.  
\end{proof} 

We also need to consider iterated Toda brackets of degree one cycles. 

\begin{lem} \label{drei} Let $1 \leq \alpha_1 \leq  \alpha_2 \leq \alpha_3$. 
\[
   \mathscr{T}(c_1, c_1, c_1) \in C(\alpha_1)_* \otimes C(\alpha_2)_* \otimes C(\alpha_3)_* \quad \textrm{and} \quad \mathscr{T}(c_1, c_1) \in C(\alpha_1)_* \otimes C(\alpha_2)_*
\]
 are positive. 
\end{lem} 

\begin{proof} 

First let  $\alpha := \alpha_1 = \alpha_2 = \alpha_3$. 
The diagonal map $\Delta : G_\alpha \to (G_\alpha)^3$, $g \mapsto (g,g,g)$, satisfies 
\[
    \Delta_*  ([c_{5}])  =  [\sum_{(d_1, d_2, d_3) \in D} c_{d_1} \otimes c_{d_2}  \otimes c_{d_3} ] \, ,
\]
where $D$ contains all triples $(d_1, d_2, d_3)$ with $0 \leq d_i \leq 5$, $\sum d_i  = 5$ and precisely one odd $d_i$. 
This follows from the ring structure of $ \HH^*(B G_\alpha ; \Z/p^\alpha)$ and the assumption that $p$ is odd.
For $1 \leq i < j \leq 3$ let $\Delta^{(i,j)} : G_{\alpha} \to (G_{\alpha})^3$ denote the diagonal map $\delta : G_{\alpha} \to (G_\alpha)^2$ composed with the embedding $(G_\alpha)^2 \to (G_{\alpha})^3$ to the $i$th and $j$th factors. 
Since $\delta_*([c_5]) = [ \sum_{i=0}^5 c_i \otimes c_{5-i} ]$ a direct calculation shows 
\[
   \Delta_*  ([c_{5}])  = [ \mathscr{T}(c_1, c_1, c_1) ]+ \sum_{1 \leq i < j \leq 3} \Delta^{(i,j)}_*([c_5]) - [c_5 \otimes c_0 \otimes c_0 + c_0 \otimes c_5 \otimes c_0 + c_0 \otimes c_0 \otimes c_5] \, . 
\]
Since $c_5$ is positive this implies positivity of $ \mathscr{T}(c_1, c_1, c_1) \in (C(\alpha)_*)^3$.
Now let $1 \leq \alpha_1 \leq \alpha_2 \leq \alpha_3$ and let 
\[
   \Phi : G_{\alpha_1} \times G_{\alpha_1}  \times G_{\alpha_1} \stackrel{ \phi_1 \times \phi_2 \times \phi_3}{\lra}  G_{\alpha_1} \times G_{\alpha_2}  \times G_{\alpha_3} \, , 
\]
where $\phi_i : G_{\alpha_1} \to G_{\alpha_i}$ is induced  by $g_{\alpha_1} \mapsto (g_{\alpha_i})^{p^{\alpha_i - \alpha_1}}$, $1 \leq i \leq 3$.
By Proposition \ref{compute} we obtain 
\[
    \Phi_* ( \mathscr{T} (c_1, c_1, c_1) )  = c_1 \otimes c_2 \otimes c_2 + p^{\alpha_2-\alpha_1} c_2 \otimes c_1 \otimes c_2 + p^{\alpha_3 - \alpha_1} c_2 \otimes c_2 \otimes c_1 =  \mathscr{T} (c_1, c_1 , c_1) \, . 
\] 
This implies the first assertion. 
The proof of the second assertion is similar.
\end{proof} 

We obtain the following conclusive result on the positivity of iterated Toda bracket cycles (which may contain degree one cycles). 
 
\begin{prop}  \label{bplreppos} Let $n \geq 2$,  $1 \leq \alpha_1 \leq \cdots \leq \alpha_n$ 
  and $m_1, \ldots, m_n \geq 1$. 
Then the following cycles in $C(\alpha_1)_* \otimes \cdots \otimes C(\alpha_n)_*$ are positive: 
\begin{enumerate}[label={(\roman*)}] 
  \item  \label{dois} $p \cdot  \mathscr{T} (c_{2m_{1} -1}^{(1)} , \ldots, c_{2m_{n} -1}^{(n)})$, if $\alpha_1 = \cdots = \alpha_n$, 
  \item  \label{um} $\mathscr{T} (c_{2m_{1}-1}^{(1)} , \ldots, c_{2m_{n} -1}^{(n)})$, if  $\alpha_{1}  < \alpha_{n}$.
\end{enumerate} 
\end{prop} 

\begin{proof}
Let $1 \leq i_1 < \ldots < i_r \leq n$ be those indices with $m_{i_j} = 1$ (where $0 \leq r \leq n$). 
Note that for all $1 \leq k \leq n$ which are different from any $i_j$ the cycle $c^{(k)}_{2m_k - 1}$ is positive. 
We consider the following cases: 
\begin{itemize} 
  \item If $r = 0$ then $\mathscr{T} (c_{2m_{1}-1}^{(1)} , \ldots, c_{2m_{n} -1}^{(n)})$ represents an iterated Toda product of positive classes and is hence positive by Corollary \ref{summary} \ref{zwoa}.
  \item If $r > 1$  then $\mathscr{T} (c_{1}^{(i_1)} , \ldots, c_{1}^{(i_r)})$ is positive by  Lemma  \ref{drei} by grouping the cycles $c^{(i_j)}_1$ for $1 \leq j \leq r$  into families of two and three, and applying the fact that Toda brackets of positive classes are positive.
  Hence $\mathscr{T} (c_{2m_{1}-1}^{(1)} , \ldots, c_{2m_{n} -1}^{(n)})$ represents an iterated Toda bracket of positive classes, and is therefore positive.
  \item If $r = 1$ let $1 \leq i \leq n$ be the unique index with $m_{i} = 1$.
If $\alpha_{1}  < \alpha_{n}$ we find $ 1 \leq k \leq n$ with $\alpha_{i} \neq \alpha_{k}$. 
Then   $\mathscr{T}(c^{(i)}_{2m_{i} - 1}, c^{(k)}_{2m_k-1})$ (resp. $\mathscr{T}(c^{(k)}_{2m_{k} - 1}, c^{(i)}_{2m_i-1})$ if   $k < i$) is positive by Proposition \ref{calculate} \ref{tre}. 
If $\alpha_{1} = \alpha_{n}$, then $p \cdot \mathscr{T} (c^{(i)}_{2m_{i} - 1}, c^{(k)}_{2m_k-1})$ is positive for any $1 \leq k \leq n$ with  $k \neq i$ by Proposition \ref{calculate} \ref{tvo}.
The proof can now be finished as before. 
\end{itemize}
\end{proof}

For $ \kappa \geq 1$, $\gamma \in \N_{>0}$  and $m = p^{\kappa}$ we get $\gamma^m \equiv \gamma \mod(p)$ in   \eqref{coprod}. 
Hence the following problem cannot be answered with the  methods developed in this section and enforces us to restrict to $p$-divisible cycles in Proposition \ref{bplreppos} \ref{dois}, in general. 

\begin{question} \label{openprob} Let $\alpha, \kappa \geq 1$. 
Is the $p$-atoral (for odd $p$) cycle $\mathscr{T}(c_1, c_{2p^{\kappa} -1}) \in C(\alpha)_* \otimes C(\alpha)_*$ positive? 
\end{question}

%%%%%
\section{Generalized products of lens spaces in group homology} \label{realis} 
%%%%%

Let $p$ be an odd prime, let $\alpha, n   \geq 1$, and let $\Gamma := (G_\alpha)^n$. 
For our proof of Theorem \ref{representable}  we will argue that certain $p$-atoral cycles  such as the one  in Question \ref{openprob} are not contained in the image of  $\Omega^{\SO}_*(B\Gamma) \to \HH_*(B\Gamma)$  and can therefore be ignored.  
In this section we will approach this issue, which is related to the classical Steenrod problem on the realization of homology classes by smooth manifolds, in terms of natural stable homology operations defined for any topological space $X$, 
\[
 \partial^{(\kappa,\ell)} : \HH_*( X ; \Z/p^{\ell})  \to  \HH_{* - 2p^{\kappa} + 1} (X ; \Z/p^{\ell}) \, ,  \kappa, \ell \geq 1 \, , 
\]
which by construction vanish on classes coming from $\Omega^{SO}_*(X)$.
In Proposition \ref{productlens}, the main result of this section, we will show that the vanishing of the operations $\partial^{(\kappa, \alpha)}$, together with the vanishing of a suitable Bockstein operation, is indeed sufficient to detect elements in the image of $\Omega^{SO}_*(B \Gamma) \to \HH_*(B \Gamma ; \F_p)$, and that these elements can be represented by products of standard $\Z/p^{\alpha}$-lens spaces. 

Since we prefer to avoid  a discussion of stable (co-)homology operations with coefficients $\Z/p^\ell$, for which we did not find a handy account in the literature, we construct the operation $\partial^{(\kappa, \ell)}$ as a differential in the Atiyah-Hirzebruch spectral sequence of a homology theory derived from Brown-Peterson theory at the prime $p$, whose well known structure allows us to derive some crucial properties of $\partial^{(\kappa, \ell)}$.

Recall that the coefficient ring for Brown-Peterson theory at the prime $p$ is isomorphic to a polynomial ring 
\[
   \BP_* \cong \Z_{(p)} [v_1, v_2, \ldots ] \, , 
\]
where $v_i \in \BP_{2p^i - 2}$. 
As usual we set $v_0 = p$. 
For $\kappa ,\ell \geq 1$ we define the ideal 
\[
    I(\kappa, \ell) := (p^{\ell}, v_1, \ldots, v_{\kappa-1}, v_{\kappa}^2, v_{\kappa +1} ,\ldots ) \subset \BP_* \, . 
\]

\begin{prop} \label{constrbpj} There is a multiplicative homology theory $\BP^{(\kappa, \ell)}$ with coefficient ring  
\[
   \BP^{(\kappa, \ell)}_* \cong \Z_{(p)} [v_1, v_2, \ldots ] / I(\kappa, \ell) 
\]
 together with a natural transformation of multiplicative homology theories $\BP \to \BP^{(\kappa, \ell)}$, which on the level of coefficients induces the projection $ \Z_{(p)} [v_1, v_2, \ldots ] \to \Z_{(p)} [v_1, v_2, \ldots ]  / I(\kappa, \ell)$. 
\end{prop} 

\begin{proof} Recall the construction of  a (homotopy) commutative ring spectrum $\BP$ representing Brown-Peterson theory for odd $p$ in \cite{Shimada}*{Corollary 6.7}, which is based on  bordism theory with Baas-Sullivan singularities killing the polynomial generators $x_j$ for $j \neq p^{i} - 1$ with $i \geq 1$, of $\pi_*(\MU)_{(p)}  \cong \Z_{(p)} [ x_j \mid j \geq 1, \deg(x_j) = 2j]$.
Here $\MU$ denotes the unitary bordism spectrum.

We construct  $\BP^{(\kappa, \ell)}$ in a similar fashion as a bordism theory with Baas-Sullivan singularities, killing the regular sequence $(p^{\ell}, x_1, \ldots, (x_{p^{\kappa}-1})^2, \ldots ) $ in $\pi_*(\MU)_{(p)}$. 
It follows from \cite{Shimada}*{Theorem 6.2} (and the assumption that $p$ is odd)  that this theory is represented by a commutative ring spectrum $\BP^{(\kappa, \ell)}$. 
Furthermore, by construction,  there  is a canonical map of ring spectra $\BP \to \BP^{(\kappa, \ell)}$ with the stated property on the level of coefficients. 
\end{proof} 

We have $\BP^{(\kappa,\ell)}_* \cong  \langle 1, v_\kappa \rangle_{\Z/p^{\ell}}$, the free graded $\Z/p^{\ell}$-module with generators $1$ in degree $0$ and $v_\kappa$ in degree $2p^{\kappa}-2$, with multiplication satisfying $v_\kappa^2 = 0$. 
The theory $\BP^{(\kappa, \ell)}$ may be considered as a form of extraordinary $K$-theory, where we have introduced the additional  truncation $v_{\kappa}^2 = 0$  for computational purposes, compare Lemma \ref{formal}. 

Let $X$ be a topological space and consider the Atiyah-Hirzebruch spectral sequence 
\[
      E^2_{s,t} = \HH_s(X ; \BP^{(\kappa, \ell)} _t) \Longrightarrow \BP^{(\kappa, \ell)}_{s+t}(X) \, . 
\]
The term $E^{2}_{s,t}$  is  non-zero precisely  for $t = 0$ and $t = 2p^\kappa -2$, and in these cases is canonically isomorphic to $ \HH_s(X ; \Z/p^{\ell})$ (depending on the choice of $v_{\kappa}$). 
In particular we have $E^{2}_{s,t} = E^{2p^{\kappa}-1}_{s,t}$ and we define
\begin{equation} \label{defpartial} 
   \partial^{(\kappa, \ell)} : \HH_*(X ; \Z/p^{\ell} ) \to \HH_{*- 2p^\kappa+1} (X; \Z/p^{\ell} ) 
\end{equation} 
as the differential $\partial^{2p^{\kappa} - 1} : E^{2p^{\kappa}-1}_{s,0} \to E^{2p^{\kappa} -1}_{s-2p^k+1, 2p^{\kappa} -2}$.
It is immediate from this construction that $\partial^{(\kappa,\ell)}$ is natural in $X$, is stable with respect to suspensions,  and is a derivation with respect to the homological cross product. 
Since the natural transformation $\Omega^{\SO}_*(X) \to  \HH_*(X ; \Z/p^\ell)$ factors through $\BP^{(\kappa, \ell)}(X)$ all classes in $\HH_*(X; \Z/p^\ell)$ coming from $\Omega^{\SO}_*(X)$ lie in the kernel of $\partial^{(\kappa, \ell)}$.  

\begin{rem} The cohomology operation $\HH^*( X ; \F_p) \to \HH^{* + 2p^{\kappa}-1} (X ; \F_p)$ dual to $\partial^{(\kappa, 1)}$ can be identified with the $\kappa$th Milnor basis element $Q_{\kappa} \in \mathscr{A}_p^{2p^{\kappa}-1}$ in the ${\rm mod}(p)$-Steenrod algebra. 
\end{rem}

For evaluating the operations $\partial^{(\kappa, \ell)}$ for $X = BG_\alpha$ we need to determine the $\BP^{(\kappa, \ell)}$-theoretic Euler class of the fibration $S^1 \hookrightarrow BG_\alpha \to \CP^{\infty}$. 
This is  based on a formal group law computation. 
For $\kappa \geq 1$ we define the ideal 
\[
   I(\kappa) := (v_1, \ldots, v_{\kappa-1} , v_{\kappa}^2, v_{\kappa+1}, \ldots) \subset \BP_* = \BP^{-*}
\]
and set $\widehat{\BP}^* := \BP^* / I(\kappa)$. 
Let $x^{\BP} \in \widetilde{\BP}^2( \CP^{\infty})$ be the standard complex orientation and recall $\BP^*(\CP^{\infty}) \cong \BP^* [[x^{\BP}]]$. 

\begin{lem} \label{formal} Let $\alpha, \kappa \geq 1$  and let  $p^\alpha : \CP^{\infty} \to \CP^{\infty}$ be the map induced by $S^1 \to S^1$, $t \mapsto t^{p^\alpha}$, using the identification $BS^1 = \CP^{\infty}$. 
Let $(p^{\alpha})^* : \BP^*(\CP^{\infty}) \to \BP^*( \CP^{\infty})$ be the induced map in $\BP$-cohomology.
Then in $\widehat{\BP}^* [[x^{\BP}]]$ we obtain the equation 
\begin{equation*} 
    (p^\alpha)^*(x^{\BP}) = p^{\alpha}  \cdot x^{\BP}  + p^{\alpha-1} \cdot v_\kappa \cdot ( x^{\BP})^{p^\kappa} + R_{\alpha} \, ,  
\end{equation*} 
where $R_{\alpha} \in p^{\alpha} \cdot \widehat{\BP}^*[[x^{\BP}]]$. 
\end{lem} 

\begin{proof} We use induction on $\alpha$. 
We write $x$ instead of $x^{\BP}$ and carry out the following computations in  $\widehat{\BP}^* [[x]]$. 
For $\alpha =1$ the $p$-typical formal group law of $\BP$ yields $(p)^*(x) = p x +  v_\kappa \cdot x^{p^\kappa}$ (possibly after multiplying $v_\kappa$ with a unit in $\Z_{(p)}$). 
Hence the assertion holds for $\alpha = 1$. 

Using $v_\kappa^2 = 0 \in \widehat{\BP}^* $ we  inductively obtain, for $\alpha \geq 1$,
\begin{align*} 
    (p^{\alpha+1})^*(x) & =  (p)^* ( (p^{\alpha})^* (x))  = \\
   & = p \cdot \big( p^{\alpha}  \cdot x + p^{\alpha-1} v_\kappa \cdot x^{p^\kappa}  + R_{\alpha} \big) + v_\kappa \cdot \big( p^{\alpha}  \cdot x + p^{\alpha-1} v_\kappa \cdot x^{p^\kappa}  + R_{\alpha}  \big)^{p^\kappa}  \\ 
   &  =  p^{\alpha+1} \cdot x + p^{\alpha} \cdot v_\kappa \cdot x^{p^\kappa} + R_{\alpha+1} \, , 
 \end{align*} 
where $R_{\alpha+1} := p\cdot R_{\alpha} + v_\kappa \cdot \big( p^{\alpha}  \cdot x + p^{\alpha-1} v_\kappa \cdot x^{p^\kappa}  + R_{\alpha}  \big)^{p^\kappa} \in p^{\alpha+1} \cdot \widehat{\BP}^*[[x^{\BP}]] $. 
\end{proof}

\begin{conv} \label{standingconv} Let  $x \in \HH^2( \CP^{\infty}; \Z)$ be the complex orientation induced by $x^{\BP}$ and for  $m \geq 0$ let $y_{2m} \in \HH_{2m}( \CP^{\infty}; \Z)$ be the generator dual to $x^m$.
For  any commutative ring $R$ with unit we obtain an  exact Gysin sequence 
\[
    \cdots  \longrightarrow \HH_*( BG_{\alpha}; R ) \stackrel{\pi_*}{\longrightarrow}  \HH_{*} ( \CP^{\infty}; R) \stackrel{- \cap e}{\longrightarrow} \HH_{* - 2} ( \CP^{\infty} ; R) \stackrel{\tau_*}{\longrightarrow}  \HH_{*-1} (  BG_{\alpha}; R ) \longrightarrow  \cdots 
\]
where $e  = p^{\alpha}  \cdot x \in \HH^2( \CP^{\infty}; \Z)$ is the Euler class of the $S^1$-principal fiber bundle 
\[
   S^1\hookrightarrow BG_{\alpha} \stackrel{\pi}{\to} \CP^{\infty}
\]
equipped with its canonical orientation induced by the inclusion $S^1 \subset \C$, and  $\tau_*$ is the homological transfer in this fiber bundle.

For $m \geq 0$ we obtain specific generators $c_{2m+1} = \tau_{2m}(y_{2m}) \in \HH_{2m+1}(BG_{\alpha}; \Z) \cong \Z/p^\alpha$ and $c_{2m} \in \HH_{2m}( BG_{\alpha} ; \Z/p^{\alpha}) \cong \Z/p^{\alpha}$ with  
\[
    \pi_*(c_{2m}) = y_{2m} \in \HH_{2m}( \CP^{\infty} ; \Z/p^\alpha).
\]
Furthermore, for $ 1 \leq \ell \leq \alpha$, these generators induce generators $c_d \in \HH_*(BG_{\alpha}; \Z/p^\ell)$ for $d \geq 0$. 

We can and will assume that the generators 
\[
   c_d \in C(\alpha)_d  = \HH_d( ( BG_\alpha)^{(d)} / (BG_\alpha)^{(d-1)}; \Z)
\]
of the cellular chain complex introduced at the beginning of Section \ref{hom_p_1} map to these  specific generators of $\HH_{d} ( BG_{\alpha}  ; \Z/p^{\ell} )$  after passing to  coefficients $\Z/p^\ell$. 
\end{conv} 

\begin{prop} \label{compute_diff} Let $\alpha , \kappa \geq 1$ and $1 \leq \ell \leq \alpha$. 
Then  the operation 
\[
   \partial^{(\kappa, \ell)} : \HH_*(BG_{\alpha} ; \Z/p^{\ell} ) \to \HH_{*- 2p^\kappa+1} (BG_{\alpha} ; \Z/p^{\ell} ),
\]
viewed as a map $C(\alpha)_* \otimes \Z/p^\ell \to C(\alpha)_{*-2p^{\kappa} +1} \otimes \Z/p^\ell$, is given by 
\[ 
  \partial^{(\kappa, \ell)}  (   c_{d}  )  =   \begin{cases} p^{\alpha-1} \cdot c_{d - 2p^\kappa +1}  &  \text{ for even }  d \geq 2p^\kappa \text{ and } \ell = \alpha\, , \\ 
                                            0 & \textrm{ otherwise } \, . 
                 \end{cases}
\]
\end{prop} 

\begin{proof} Using the canonical isomorphism   
\[
   \HH_*(\CP^{\infty} ; \BP_*^{(\kappa, \ell)}( S^1) )  \cong \BP^{(\kappa, \ell)}_*(\CP^{\infty}; \HH_*(S^1))
\]
 we may write the $\BP^{(\kappa, \ell)}$-homology spectral sequence for the fiber bundle  $S^1 \hookrightarrow BG_{\alpha} \to \CP^{\infty}$  as
\[
    E^2_{s,t} =  \BP^{(\kappa, \ell)}_s(\CP^{\infty}; \HH_t(S^1))  \Longrightarrow  \BP^{(\kappa, \ell)}_{s+t}( B G_{\alpha} )  \, .  
\]
The only nonvanishing differential  is given by $\partial^2 : E^2_{s,0} \to E^2_{s-2, 1}$, $c \mapsto  (c \cap e) \otimes [S^1] $, where  $c \in \BP_s^{(\kappa, \ell)}(\CP^{\infty})$, $e \in ( \BP^{(\kappa, \ell )})^2 ( \CP^{\infty})$ is the  $\BP^{(\kappa, \ell)}$ theoretic Euler class of $S^1 \hookrightarrow BG_{\alpha} \to \CP^{\infty}$, and $[S^1] \in \HH_1(S^1 ; \Z)$ is the given orientation class. 
Note that this  spectral sequence induces the $\BP^{(\kappa, \ell)}$-theoretic Gysin sequence for the fiber bundle $S^1 \hookrightarrow BG_{\alpha} \to \CP^{\infty}$. 

Let $x \in (\BP^{(\kappa, \ell)})^2(\CP^{\infty})$ be the class induced by $x^{\BP}$. 
Viewing $v_{\kappa} \in  \BP^{(\kappa, \ell)}_{2p^{\kappa} - 2}  \cong (  \BP^{(\kappa, \ell)} )^{-2p^{\kappa} + 2}$ Lemma \ref{formal} implies 
\begin{equation} \label{euler} 
      e  = \begin{cases} p^{\alpha-1} \cdot v_\kappa \cdot x^{p^\kappa}  \mbox{ for } \ell = \alpha \, , \\ 
                                                                                                                  0 \mbox{ for } 1 \leq \ell < \alpha \, . \end{cases}  
\end{equation}                    
Now consider  the isomorphisms
\begin{align*} 
    E^2_{*,0}  & \cong  \HH_{*}  (\CP^\infty ; \BP^{(\kappa, \ell)}_{\rm ev} ( S^1) ) \cong   \HH_{\rm ev}(BG_\alpha; \Z/p^{\ell}) \otimes \langle 1, v_{\kappa} \rangle_{\Z/p^{\ell}}  \, , \\
    E^2_{*, 1} & \cong   \HH_* (\CP^\infty ; \BP^{(\kappa, \ell)}_{\rm odd} ( S^1 ) ) \cong   \HH_{\rm odd}(BG_\alpha; \Z/p^{\ell}) \otimes \langle 1, v_{\kappa} \rangle_{\Z/p^{\ell}} \, , 
 \end{align*} 
 the first one of which is induced by the projection $BG_{\alpha}\to \CP^{\infty}$ and the second one by the homological transfer $\tau_*$ for the bundle $S^1 \hookrightarrow BG_{\alpha} \to \CP^{\infty}$. 
 Under these isomorphisms the differential $\partial^2 : E^2_{s,0} \to E^2_{s-2,1}$, $c \mapsto (c \cap e) \otimes [S^1]$, corresponds to the differential 
 \[
      \partial^{2p^{\kappa} -1} :   \HH_{s}( BG_{\alpha} ;  \BP^{(\kappa, \ell)}_t) \to \HH_{s - 2p^{\kappa} +1} ( BG_\alpha; \BP^{(\kappa, \ell)}_{t+2p^{\kappa} - 2}  ) 
 \]
in the Atiyah-Hirzebruch spectral sequence $E^2_{s,t} = \HH_{s}(BG_{\alpha} ; \BP_t^{(\kappa, \ell)} )  \Longrightarrow \BP^{(\kappa, \ell)}_{s+t} ( BG_\alpha)$.
Since this differential defines $\partial^{(\kappa, \ell)}$, Proposition \ref{compute_diff} follows. 
\end{proof} 

\begin{ex}  \label{beispiel} Let $\alpha, \kappa \geq 1$ and $\Gamma = (BG_{\alpha})^2$. 
Then in $\HH_*( B\Gamma ; \Z/p^{\alpha})$ we get
\[
    \partial^{(\kappa, \alpha)} ( \mathscr{T}(c_1, c_{2p^{\kappa} -1}) )  =  \partial^{(\kappa, \alpha)} ( c_2 \otimes c_{2p^{\kappa}-1} + c_1 \otimes c_{2p^{\kappa}} ) = p^{\alpha-1} \cdot (c_1 \otimes c_1 ) \neq 0, 
 \]
and hence the cycle $\mathscr{T}(c_1, c_{2p^{\kappa} -1} )  \in C(\alpha)_* \otimes C(\alpha)_*$ appearing in Question \ref{openprob} does not lift to $\Omega^{\SO}_{2p^{\kappa} +1} (B \Gamma)$. 
For $p = 3$ and $\alpha, \kappa = 1$ this reproduces the class in $\HH_7( B (\Z/3)^2 ; \Z)$ considered  in  \cite{Thom54}*{page 62}, which was the first example of an integral homology class  that cannot be represented by a smooth manifold. 
\end{ex} 

For a topological space $X$ and $\ell \geq 1$ we denote by $\beta^{(\ell)} :\HH_*(X  ; \Z/p^{\ell}) \to \HH_{* -1}(X ; \F_p)$ the Bockstein operation for the exact coefficient sequence $0 \to \Z/p \stackrel{\cdot p^{\ell}}{\rightarrow} \Z/p^{\ell+1} \to \Z/p^\ell \to 0$. 
Note that $\beta^{(\ell)}$ vanishes on classes that lift to integral homology. 

\begin{defn} \label{RH} Let $\ell \geq 1$. 
We call the submodule
\[
    \RH_*( X ; \Z/p^\ell) :=\ker \beta^{(\ell)}  \, \cap \, \bigcap_{\kappa \geq 1} \,   \ker  \partial^{(\kappa, \ell)}  \subset \HH_*(X ; \Z/p^\ell)  
\]
 the {\em almost representable homology} in $\HH_*(X ; \Z/p^\ell)$.
\end{defn} 

Let $\alpha, n  \geq 1$ and $\Gamma = (G_{\alpha})^n$.
It follows from Proposition \ref{compute_diff}  that  $p \cdot \HH_*( B\Gamma; \Z/p^\alpha)$ is contained in $\RH_*(B\Gamma; \Z/p^{\alpha})$. 
The same holds for the image of $\Omega^{\SO}_*(B \Gamma) \to \HH_*(B\Gamma; \Z/p^{\alpha})$. 
In Proposition \ref{productlens} we will show a weak converse of the last statement.
We first define specific elements  in the group homology $\HH_*(B \Gamma; \F_p)$, which are represented by smooth manifolds. 

\begin{defn} For $m \geq 1$ denote by $L^{2m-1} = S^{2m-1} / (\Z/p^{\alpha})$ the standard $\Z/p^{\alpha}$-lens space. 
Let $1 \leq k \leq n$ and let $\phi: (G_{\alpha})^k \to (G_{\alpha})^n$ be some group homomorphism. 
For $m_1, \ldots, m_k \geq 1$ we obtain the map 
\[
   \Phi:   L^{2m_1-1} \times \cdots \times L^{2m_k-1} \stackrel{\Psi}{\lra}  B(G_\alpha)^k \stackrel{B\phi}{\longrightarrow}  B(G_{\alpha})^n = B \Gamma\, , 
\]
where  $\Psi$ is the product of classifying maps. 

The class $\Phi_*( [ L^{2m_1-1} \times \cdots \times L^{2m_k-1}] ) \in \HH_*(B \Gamma; \F_p)$  is called a {\em generalized product of lens spaces}. 
Obviously this element lifts to $\Omega^{\SO}_*(B \Gamma)$. 
\end{defn} 

We can now state the main result of this section. 

\begin{prop} \label{productlens} The image of  $\RH_*( B \Gamma ; \Z/p^\alpha) \hookrightarrow \HH_*(B \Gamma ; \Z/p^{\alpha}) \to \HH_*(B\Gamma; \F_p)$  is  generated by generalized products of $\Z/p^{\alpha}$-lens spaces. 
\end{prop} 

\begin{rem} It was observed first in \cite{BR1}*{Theorem 5.6} that generalized products of lens spaces generate the image of $\Omega^{\SO}_*( B (\Z/p)^n) \to  \HH_*(B(\Z/p)^n; \F_p)$. 
A complete  proof of this statement was given in \cite{Ha15}.
Already for $\alpha=1$ Proposition \ref{productlens} is stronger than \cite{BR1}*{Theorem 5.6} as it is not clear a priori that all classes in $ \RH_*( B(\Z/p)^n ; \F_p)$ lift to $\Omega^{\SO}_*(B (\Z/p)^n)$.
\end{rem}

The proof of Proposition \ref{productlens}, which will be given at the end of this section, requires some preparation. 
Our argument is mainly algebraic and in principle carried out in  the reduced homology $\tilde \HH_*( \widehat{B\Gamma} ; \F_p)$, which is reflected in the following notation. 
Let 
\begin{itemize} 
   \item $C_* = \tilde \HH_*( B G_{\alpha} ; \F_p)$ be  the free $\Z$-graded $\F_p$-module with one generator $c_d$ in each degree $d\geq 1$;
   \item $(C_*)^n = \tilde \HH_*( \widehat{B\Gamma} ; \F_p)$ be its $n$-fold tensor product,  with $n \geq 0$; 
   \item  $\partial^{(\kappa)}$, $ \kappa \geq 0$,  be the differential on $(C_*)^n$ of degree $-2p^{\kappa} +1$, which acts as a derivation and satisfies 
\[ 
   \partial^{(\kappa)}   (c_{d}) :=  \begin{cases} c_{d - 2p^\kappa +1} &  \text{ for even }  d  \geq  2p^\kappa - 1 \, , \\
      0 & \text{ otherwise}  \, ;  \end{cases} 
\]
   \item $\mathscr{C}^{n,r}_* := \bigcap_{0 \leq \kappa \leq r} \ker \partial^{(\kappa)} \subset (C_*)^n$ for $r \geq 0$, and $\mathscr{C}^{n,\infty}_*  = \bigcap_{\kappa \geq 0} \ker \partial^{(\kappa)}$.
  \end{itemize}

\begin{prop} \label{impliso} The canonical map  
\[
   (C(\alpha)_* \otimes \Z/p^\alpha)^n \to (\tilde C(\alpha)_*\otimes \F_p )^n = (C_*)^n
\]
sends $ \RH_*(B\Gamma; \Z/p^\alpha)$ onto $ \mathscr{C}^{n,\infty}$. 
\end{prop} 

\begin{proof} The Bockstein operation $\beta^{(\alpha)} : \HH_*(BG_\alpha; \Z/p^{\alpha}) \to \HH_{*-1}(BG_\alpha; \F_p)$ is given by  $C(\alpha)_* \otimes \Z/p^\alpha \to C(\alpha)_{*-1} \otimes \F_p$, $c_0 \mapsto 0$, $c_{2m-1} \mapsto 0$, and $c_{2m} \mapsto c_{2m-1}$ for $ m \geq 1$. 
Hence its $n$-fold tensor product derivation restricts to a  map $ (\tilde C(\alpha)_* \otimes \Z/p^\alpha)^n \to (C_*)^n$ whose kernel goes onto $\ker \partial^{(0)} \subset (C_*)^n$ under tensoring the domain with $\F_p$. 

For $\kappa \geq 1$ the computation of $ \partial^{(\kappa, \alpha)} : \HH_*(BG_{\alpha} ; \Z/p^{\alpha} ) \to \HH_{*- 2p^\kappa+1} (BG_{\alpha} ; \Z/p^{\alpha} )$ in Proposition \ref{compute_diff}   implies that its $n$-fold tensor product derivation restricts to a map $ (\tilde C(\alpha)_* \otimes \Z/p^\alpha)^n \to (\tilde C(\alpha)_*\otimes \Z/p^\alpha)^n$  whose kernel goes onto $\ker \partial^{(\kappa)} \subset (C_*)^n$ under tensoring the domain with $\F_p$. 

From these facts  Proposition \ref{impliso}  follows, using $p \cdot \HH_*( B\Gamma; \Z/p^\alpha) \subset \RH_*( B\Gamma ; \Z/p^\alpha)$. 
\end{proof}

We will now analyse the submodule $\mathscr{C}^{n,\infty}_* \subset (C_*)^n$. 
For this aim we define the $\Z$-graded $\F_p$-modules: 
\begin{itemize} 
 \item $N_* := {\rm span} \{ c_{2m-1} \mid m \geq 1 \}  = \tilde \HH_{\rm odd} ( BG_\alpha; \F_p)  \subset C_*$;
 \item $L_* := {\rm span} \{ y_{2m}  \mid m \geq 1 \} = \tilde \HH_{\rm even} ( \CP^{\infty} ; \F_p)$, where $y_{2m}$ are free generators of degree $2m$;
   \item $L_{< p^k} := {\rm span} \{ y_{2m} \mid 1 \leq m < p^k\} \subset L_*$ for $k \geq 0$.
 \end{itemize} 
Note that the canonical projection $C_* \to L_*$,   $c_{2m} \mapsto y_{2m}$, $c_{2m-1} \mapsto 0$ (which on the topological side is induced by $BG_\alpha \to \CP^{\infty}$) commutes with the differentials $\partial^{(\kappa)}$ for $\kappa \geq 0$, which we define as zero on $L_*$. 

Let  $(N_*)^n$ be the $n$-fold tensor product of $N_*$ for $n \geq 0$.
For every $1 \leq k \leq n$ and every group homomorphism $\phi : (G_\alpha)^k \to (G_\alpha)^n = \Gamma$ we obtain an induced map 
\begin{equation} \label{cool} 
 \phi_* : (N_*)^k   \hookrightarrow  \tilde \HH_*( B(G_\alpha)^k  ; \F_p) \stackrel{(B\phi)_*}{\longrightarrow}  \tilde \HH_* (B\Gamma ; \F_p)  \longrightarrow \tilde \HH_*( \widehat{ B\Gamma}; \F_p)  = (C_*)^n \, . 
\end{equation}

\begin{defn} \label{curll} For $n \geq 1$ we set 
\[
   \mathscr{L}_*^{n} :=  {\rm span} \{  \phi_*\big((N_*)^k\big) \mid \phi : (G_\alpha)^k \to (G_\alpha)^n  \text{ group homomorphism}, 1 \leq k \leq n\} \subset (C_*)^n \, . 
\]
\end{defn} 
Since the generators of $N_*$ are represented by $\Z/p^\alpha$-lens spaces we have $\mathscr{L}_*^{n} \subset \mathscr{C}^{n,\infty}_*$ by Proposition \ref{impliso}.
The crucial step for the proof of Proposition \ref{productlens} consists in showing that here equality holds,  see Proposition  \ref{description}.
We first derive a lower bound for the size of $\mathscr{L}_*^{n} \subset \mathscr{C}^{n,\infty}_*$. 

\begin{prop} \label{Vandermonde} For $n \geq 1$ the canonical projection $\mathscr{L}_*^{n+1} \to (N_*)^n \otimes L_{< p^n}$ is surjective.
\end{prop}   

\begin{proof} 
Essentially the proof for  $\alpha = 1$ in \cite{Ha15}*{Proposition 5.3} generalizes to larger $\alpha$. 
For notational reasons we work with the additive group $\Z/p^{\alpha}$ instead of $G_{\alpha}$.

For $0 \leq \lambda_1, \ldots, \lambda_n \leq p-1$ we consider the group homomorphism 
\begin{eqnarray*} 
  \phi_{(\lambda_1, \ldots, \lambda_n)}   :  (\Z/p^\alpha)^n & \to & (\Z/p^\alpha)^{n+1} \\ 
      (x_1, \ldots, x_n) & \mapsto & ( x_1, \ldots, x_n, \lambda_1 x_1 + \cdots + \lambda_n x_n) \, . 
\end{eqnarray*} 
For all $\gamma \geq 1$ we have an $\F_p$-algebra isomorphism
\[
    \HH^*(B(\Z/p^\alpha)^\gamma ;\F_p) \cong  \F_p[t_1, \ldots, t_\gamma] \otimes \Lambda(s_1, \ldots, s_\gamma) \, , 
\]
where $t_1, \ldots, t_\gamma$ are indeterminates of degree $2$ and $s_1, \ldots, s_\gamma$ are indeterminates of degree $1$. 

The map induced in $\F_p$-cohomology  by $B \phi_{(\lambda_1, \ldots, \lambda_n)} : B (\Z/p^\alpha)^n \to B(\Z /p^{\alpha})^{n+1}$ satisfies
\begin{equation} \label{firsteq} 
    (  t_1^{m_1} s_1 \cdot  \ldots \cdot  t_n^{m_n} s_n ) \cdot t_{n+1}^{\nu} \mapsto (t_1^{m_1} s_1 \cdot \ldots \cdot t_n^{m_n} s_n) \cdot (\lambda_1 t_1 + \cdots + \lambda_n t_n)^{\nu} 
\end{equation} 
for $\nu \geq 0$. 
This computation uses the ring structures of $ \HH^*(B(\Z/p^\alpha)^\gamma ;\F_p)$ for $\gamma = n , n+1$. 

The $p^n \times p^n$ Vandermonde-matrix  
 \[ 
    V :=   \left( \begin{array}{cccc} 1 &  \lambda_1 t_1 + \cdots + \lambda_n t_n  & \cdots & (\lambda_1 t_1 + \cdots + \lambda_n t_n)^{p^n-1} \end{array} \right)_{ 0 \leq \lambda_1, \ldots , \lambda_n < p}
 \]
 (where the subscript parametrizes the rows) with entries in $\F_p[t_1, \ldots, t_n]$ has  determinant 
 \[ 
     \prod_{(\lambda_1, \ldots, \lambda_n) < (\mu_1, \ldots, \mu_n)} \big((\mu_1 - \lambda_1)t_1 + \cdots + ( \mu_n - \lambda_n)t_n\big) \neq 0,
 \]
applying  the lexicographic order to the index set. 
 Hence the column vectors of $V$ are linearly independent over $\F_p[t_1, \ldots, t_n]$. 
  
Setting $N^*  : = \HH^{\rm odd}(B\Z/p^{\alpha}; \F_p)$ this means, in view of \eqref{firsteq}, that the map  
\[
\bigoplus_{ 0 \leq \lambda_1, \ldots, \lambda_n < p}  \phi_{(\lambda_1, \ldots, \lambda_n)}^* :  (N^*)^n   \otimes \HH^{0 \leq 2m < 2p^n} (B\Z/p^{\alpha};\F_p) \longrightarrow 
\bigoplus_{0 \leq \lambda_1, \ldots, \lambda_n < p} (N^*)^n  
\]
is injective.  
Dualizing this statement over $\F_p$ we conclude that the map 
\[
    \sum_{ 0 \leq \lambda_1, \ldots, \lambda_n < p } (  \phi_{\lambda_1, \ldots, \lambda_n})_* :    \bigoplus_{0 \leq \lambda_1, \ldots, \lambda_n < p} (N_*)^n  \to   (N_*)^n \otimes {\rm span}_{\F_p} \{  c_0, \ldots, c_{2(p^n-1)} \} 
\]
is  surjective.
\end{proof} 

The modules $ (N_*)^n \otimes L_{< p^n}$ play an important role in the determination of $\mathscr{C}^{n, \infty}$, which we will carry out in two steps. 
First note that we have a canonical direct sum decomposition
\[
      (C_*)^n = \bigoplus_{\gamma = 0}^{n} ( C_*)^{n}_{(\gamma)} \, , 
\]
where $(C_*)^n_{(\gamma)} \subset  (C_*)^n$ is generated by those elementary tensors $c_{d_1} \otimes \cdots \otimes c_{d_n}$ involving $\gamma$ components of even degree.
For example $(C_*)^n_{(0)} = (N_*)^n$. 
Since the differentials $\partial^{(\kappa)}$ map $(C_*)^n_{(\gamma)}$ to $(C_*)^n_{(\gamma-1)}$ we get induced direct sum decompositions of $\mathscr{C}_*^{n, r}$ for $r \geq 0$. 

The next, somewhat involved, proposition takes care of the particular component 
\[
    \mathscr{D}_*^{n,r} := \mathscr{C}^{n,r}_* \cap (C_*)^n_{(1)} \subset  \mathscr{C}_*^{n,r}
\]
for certain $r$. 
The full structure of $ \mathscr{C}^{n,\infty}_*$ will afterwards be determined in  Proposition  \ref{description}. 
Note that for $r \geq 0$ the differential $\partial^{(r)} : (C_*)^n \to (C_*)^n$ induces a map $\partial^{(r)} :  \mathscr{D}_*^{n,r-1} \to (N_*)^n$ with kernel $\mathscr{D}^{n,r}_*$. 
Here and later we set $ \mathscr{D}_*^{n,-1} := (C_*)^n_{(1)}$.

\begin{prop}  \label{difficult} For $n \geq 0$ the following holds. 
\begin{enumerate}[label={(\roman*)}]  
\item \label{erst} The canonical projection $\pi : \mathscr{D}^{n+1,n-1}_*  \to (N_*)^n \otimes L_*$ is surjective and there exists a surjective map $\overline{\partial^{(n)}} : (N_*)^n \otimes L_* \to (N_*)^{n+1}$ such that the following diagram commutes: 
\[
   \xymatrix{ \mathscr{D}^{n+1,n-1}_*    \ar@{>}[rr]^{\pi} \ar@{>}[dr]^{\partial^{(n)}} &   & (N_*)^n \otimes L_* \ar@{>}[ld]_{\overline{\partial^{(n)}}}   \\ 
  &      ( N_*)^{n+1}    &    } 
\]
\item \label{zweit} The projection $\ker(  \overline{\partial^{(n)}}) \to (N_*)^n \otimes L_{< p^n}$ is an isomorphism. 
\item \label{dritt} $\mathscr{D}^{n+1,n}_* ( = \ker (\partial^{(n)}) )= \mathscr{D}^{n+1, \infty}_* $. 
\end{enumerate} 
\end{prop}

\begin{proof} We apply induction on $n$. 
For  $n = 0$ the proposition holds as
\begin{itemize} 
   \item $\mathscr{D}^{1,-1}_* = L_* = (N_*)^0 \otimes L_*$ and $\pi$ is an isomorphism,
   \item $\partial^{(0)} :  \mathscr{D}^{1,-1}_* \to N_{*-1}$ is an isomorphism and hence $\mathscr{D}^{1,0}_* = 0 = (N_*)^0 \otimes L_{< p^0}$. 
 \end{itemize} 
Now assume  that $n \geq 1$ and Proposition \ref{difficult} has been shown up to $n -1$.
Let $ c  = c_{2d_1 -1} \otimes \cdots \otimes c_{2d_n-1} \in  (N_*)^{n}$ and let $m > 0$. 
We will show $c \otimes y_{2m} \in \im ( \pi) $. 
Let $0 \leq \kappa \leq n-1$. 
Using the inductive assumption \eqref{erst}  we find  $\overline{c(\kappa)} \in \mathscr{D}^{\kappa+1, \kappa-1}_{ *}$ with $\partial^{(\kappa)}(\overline{c({\kappa})}) = c_{2d_1-1} \otimes \cdots \otimes c_{2d_{\kappa+1} - 1}$.
Setting $c(\kappa) := \overline{c(\kappa)} \otimes c_{2d_{\kappa+2} -1} \otimes \cdots \otimes c_{2d_n - 1}$ we then have $c(\kappa) \in \mathscr{D}^{n, \kappa-1}_{ \deg(c) + 2p^{\kappa}-1}$ and $\partial^{(\kappa)} (c(\kappa)) = c$. 
Using the induction assumption again several times  in order to balance $\partial^{(j)} (c(\kappa))$ for $j = \kappa+1, \ldots, n-1$ we can arrange furthermore  that   $\partial^{(j)} (c(\kappa)) = 0$ for $\kappa < j \leq n-1$. 
Summarizing we have $\partial^{(j)}(c(\kappa)) = 0$ for  $0 \leq j \leq n-1$ with $j \neq \kappa$, while $\partial^{(\kappa)}(c(\kappa)) = c$. 

With these choices we get 
\begin{equation} \label{preimage} 
    c \otimes c_{2m}  + (-1)^{n+1} \cdot \sum_{\kappa=0}^{n-1} c(\kappa) \otimes c_{2m - 2p^{\kappa}+1} \in \mathscr{D}^{n+1,n-1}_{*} 
\end{equation} 
and $\pi$ indeed  sends this element  to $c \otimes y_{2m} \in (N_*)^n \otimes L_*$.
This shows surjectivity of $\pi$.

If $c \in \mathscr{D}^{n+1, n -1}_{*} \cap \ker ( \pi)$, then $c \in  \mathscr{D}^{n, n -1}_{*} \otimes N_*$ by the definition of  $ \mathscr{D}^{n, n -1}_{*}$ and hence  $c \in \mathscr{D}^{n, \infty}_{*} \otimes N_*$, using the inductive assumption \eqref{dritt}. 
We conclude $\partial^{(n)}(c) = 0$, and hence $\overline{\partial^{(n)}}$ is well defined.

Next let  $c\in (N_*)^n$ and let  $m > 0$. 
We claim $c \otimes c_{2m-1} \in \im ( \partial^{(n)})$, showing that $\partial^{(n)}$, and hence $\overline{\partial^{(n)}}$, is surjective. 
The proof is by induction on $\deg(c)$.
As in  \eqref{preimage} we find  $c(\kappa) \in \mathscr{D}^{n, \kappa-1}_{ \deg(c) + 2p^{\kappa}-1}$ for $0 \leq \kappa \leq n-1$  with 
\[
    c \otimes c_{2m+2p^n-2}  + ( -1)^{n+1} \cdot \sum_{\kappa=0}^{n-1} c(\kappa) \otimes c_{(2m + 2p^n - 2) - 2p^{\kappa} +1} \in \mathscr{D}^{n+1,n-1}_{*}  \, . 
\]
We have  $\partial^{(n)} ( c \otimes c_{2m+2p^n-2} ) = (-1)^n c \otimes c_{2m-1}$ and 
 \[ 
   \partial^{(n)} \big( \sum_{\kappa=0}^{n-1} c(\kappa)  \otimes c_{2m + 2p^n-2p^{\kappa} -1} \big)  =  \sum_{\kappa=0}^{n-1} \partial^{(n)}_*  ( c(\kappa) ) \otimes c_{2m + 2p^n-2p^{\kappa} - 1}  \, .  
 \]
For $0 \leq \kappa\leq n-1$ we compute
\[
    \deg \big(  \partial^{(n)}  (c(\kappa))   \big) = \deg (c(\kappa) ) - ( 2p^n -1) = (  \deg( c) + 2p^{\kappa} - 1) - (2p^n  -1) < \deg(c) \, .
\] 
Hence $\partial^{(n)}_*  ( c(\kappa) ) \otimes c_{2m + 2p^n-2p^{\kappa} - 1} \in \im ( \partial^{(n)})$ by induction on $\deg(c)$. 
Altogether we see $c \otimes c_{2m-1} \in \im ( \partial^{(n)})$ as required. 
The proof  of \ref{erst}, the most difficult part of Proposition \ref{difficult},   is now complete. 

For \ref{zweit}  we first observe that  $\dim \ker ( \overline{\partial^{(n)}})_d = \dim \big( (N_*)^n \otimes L_{< p^n}\big)_d$ for $d \geq 0$, since, by an elementary  dimension count, 
\[
  \dim \big( (N_*)^n \otimes L_* \big)_d =   \dim\big(  (N_*)^n \otimes L_{< p^n} \big)_d  +  \dim \big( (N_*)^{n+1}\big)_{d - 2p^n +1}  
\]
and $\overline{\partial^{(n)}}$ is surjective by \ref{erst}. 
Furthermore 
\[
   \mathscr{L}_*^{n+1} \subset  \mathscr{C}^{n+1, n}_* \stackrel{{\rm proj.}}{\lra} \mathscr{D}^{n+1, n}_*  \stackrel{\pi}{\longrightarrow}  (N_*)^n \otimes L_* \longrightarrow (N_*)^n \otimes L_{< p^n}
\]
 is surjective by Lemma \ref{Vandermonde}, so that also  the projection $\ker ( \overline{\partial^{(n)}}) \to (N_*)^n \otimes L_{< p^n}$ is surjective. 
Since domain and target of this map have the same dimension in each degree  this implies assertion \ref{zweit}. 

For assertion \ref{dritt}  let $c \in  \mathscr{D}^{n+1, n}_{*}$. 
Since $\mathscr{L}_*^{n+1} \subset  \mathscr{C}^{n+1,\infty}_*$ Proposition \ref{Vandermonde} implies that there exists $x \in \mathscr{D}^{n+1, \infty}_{*}$  such that the projection of $c + x$ to $(N_*)^n \otimes L_{ < p^n}$ vanishes, 
Since $\partial^{(n)}_*  ( c+x) = 0$ and $\ker ( \overline{\partial^{(n)}})$ maps isomorphically to $(N_*)^n \otimes L_{< p^n}$ we obtain $\pi(c+x) =0 \in (N_*)^n \otimes L_*$. 
We conclude $c + x  \in \mathscr{D}^{n, n}_{*}  \otimes N_* = \mathscr{D}^{n, \infty}_{*}  \otimes N_*$ by the induction assumption \ref{dritt}. 
Since  $\mathscr{D}^{n, \infty}_{*}  \otimes N_* \subset  \mathscr{D}^{n+1, \infty}_{*}$ and $x \in \mathscr{D}^{n+1, \infty}_{*}$ assertion \ref{dritt}  follows. 
 \end{proof}

\begin{cor} \label{clear}  $\ker  \big(  \mathscr{C}^{n+1,\infty}_*  \to (N_*)^n \otimes L_{< p^n} \big) \subset  \ker \big(   \mathscr{C}^{n+1, \infty}_*  \to (N_*)^n \otimes L_{*} \big) $. 
\end{cor} 

\begin{proof}  Let $c \in \ker  ( \mathscr{C}^{n+1,\infty}_*  \to (N_*)^n \otimes L_{< p^n} ) $. 
We decompose $c = c' + c''$ where $c' \in \mathscr{D}^{n+1, \infty}_*$ and $c''$ is a linear combination of elementary tensors $c_{d_1} \otimes \cdots \otimes c_{d_{n+1}}$ with $0$ or at least $2$ even degree components. 

Obviously $c'' \in \ker ( \mathscr{C}^{n+1, \infty}  \to (N_*)^n \otimes L_* \big)$.
This fact and the assumption on $c$ imply $c' \in \ker ( \mathscr{D}^{n+1, \infty}_* \to (N_*)^n \otimes L_{< p^n}) $. 
But  by Proposition \ref{difficult} \ref{zweit} the projection $(N_*)^n \otimes L_* \to (N_*)^{n} \otimes L_{< p^n}$ induces an isomorphism $\pi (\mathscr{D}^{n+1, n}_* ) \cong (N_*)^{n} \otimes L_{< p^n}$. 
Hence we must also have $c' \in \ker ( \mathscr{D}^{n+1,\infty}_* \to  (N_*)^{n} \otimes L_{*})$. 
 \end{proof} 
 
 We finally obtain a precise description of $ \mathscr{C}^{n,\infty}_*$ and of $ \mathscr{L}^{n}_* \subset \mathscr{C}^{n,\infty}_*$ (recall Definition \ref{curll}), showing in particular that the last inclusion is an equality. 

\begin{prop} \label{description} Let $\mathscr{J}_n$ denote the set  of families $ J = (J_1, \ldots, J_n)$, where $J_i = N_*$ or $J_i = L_{< p ^k}$ and $k$ is the number of $J_j$ for  $j < i$  with $J_j = N_*$.  
Then the canonical map (induced by projections $C_* \to N_*$ and $C_* \to L_*$) 
\[
      \Psi^n :   \mathscr{C}^{n,\infty}_*  \rightarrow \bigoplus_{\mathscr{J}_n}  J_1 \otimes \cdots \otimes J_n
\]
is an isomorphism.
The restriction of $\Psi^n$ to $\mathscr{L}^n_* \subset   \mathscr{C}^{n,\infty}_*$ is still surjective, and hence also an isomorphism. 
In particular $\mathscr{L}^n_* = \mathscr{C}^{n, \infty}_*$.
\end{prop}

\begin{proof}
Since source and target of $\Psi^1$ are equal to $\mathscr{L}^1_* = N_*$ the assertions are clear for $n=1$. 
  By induction we assume that they hold for some $n \geq 1$. 

For $J = (J_1, \ldots, J_n) \in \mathscr{J}_n$ let $k(J)$ denote the number of components $J_i = N_*$. 
Furthermore we set $J'_i = C_*$ for $J_i = N_*$ and $J'_i  = L_{< p^k}$ for $J_i = L_{< p^k}$. 
In  the induction step we first prove injectivity of $\Psi^{n+1}$. 

Let  $c \in  \ker \Psi^{n+1}$. 
We study the image of $c$ under  the composition of projections 
\[
    \mathscr{C}^{n+1, \infty}_*  \stackrel{\pi_1}{\longrightarrow}   \Big( \bigoplus_{\mathscr{J}_n} J'_1 \otimes \cdots \otimes J'_n \Big) \otimes C_* \stackrel{\pi_2}{\longrightarrow}   \Big(  \bigoplus_{\mathscr{J}_n}   J_1 \otimes \cdots \otimes J_{n} \Big) \otimes L_* \, . 
\]
Note that $\pi_1$ commutes with the differentials $\partial^{(\kappa)}$ for $\kappa \geq 0$ (with zero differentials on $L_{< p^k}$). 

Let  $c'\in J'_1 \otimes \cdots \otimes J'_n \otimes C_*$ be one component of $\pi_1(c) $, where $J \in \mathscr{J}_n$. 
 By assumption  the image of $c'$ under the map 
 \[
     J'_1  \otimes \cdots \otimes J'_n \otimes C_* \to J_1 \otimes \cdots \otimes J_n \otimes L_{< p^{k(J)}}
  \]
is zero. 
By Corollary \ref{clear} we have 
\[
     \ker  \big(  \mathscr{C}^{k(J) +1, \infty}_{*}  \to (N_*)^{k(J)}  \otimes L_{< p^{k(J)}} \big) \subset  \ker \big(   \mathscr{C}^{k(J) +1, \infty}_{*}   \to (N_*)^{k(J)}   \otimes L_{*} \big)
\]
and hence $\pi_2(c') = 0$ (recall that $k(J) +1$ is the number of factors $C_*$ in $J'_1  \otimes \cdots \otimes J'_n \otimes C_*$ and that all other factors are equal to some $L_{<p^k}$ with zero differential). 
Applying this argument to all components $c'$ of $\pi_1(c)$ we conclude $\pi_2( \pi_1(c)) = 0$. 

Let $\pi  : \mathscr{C}^{n+1, \infty}_* \to \mathscr{C}^{n, \infty}_* \otimes L_*$ be the projection (recall again that $C_* \to L_*$ commutes with all differentials $\partial^{(\kappa)}$ for $\kappa \geq 0$). 
Since $( \Psi^n \otimes \id) \circ \pi = \pi_2 \circ \pi_1$ our induction assumption (injectivity of $\Psi^n$)  implies $\pi(c) = 0 $. 
Hence $c \in \ker \Psi^n \otimes N_*$, which is equal to $0$, again by the induction assumption.
This shows that $\Psi^{n+1}$ is injective.  

We next show that $\Psi^{n+1}$ maps $\mathscr{L}^{n+1}_* \subset \mathscr{C}^{n+1, \infty}_*$ surjectively onto  $\bigoplus_{\mathscr{J}_{n+1}}  J_1 \otimes \cdots \otimes J_{n+1}$, completing the induction step. 

Let $(J_1, \ldots, J_{n+1}) \in \mathscr{J}_{n+1}$. 
We have to show $J_1 \otimes \cdots \otimes J_{n+1} \subset \Psi^{n+1} ( \mathscr{L}_*^{n+1})$. 
 By induction we have $ J_1 \otimes \cdots \otimes J_n \subset \Psi^n ( \mathscr{L}_*^{n})$. 
 In particular $J_1 \otimes \cdots \otimes J_n \otimes N_* \subset \Psi^{n+1} ( \mathscr{L}_*^n \otimes N_*)$. 
Since $\mathscr{L}_*^n \otimes N_* \subset \mathscr{L}_*^{n+1}$ we can hence restrict to the case  $J_{n+1} = L_{< p^{k(J)}}$, where $J:= (J_1, \ldots, J_n)$.
For each group homomorphism $\phi : (G_\alpha)^{k(J)} \to (G_\alpha)^{k(J)+1}$ we obtain an induced map $(N_*)^{k(J)} \to (N_*)^{k(J)} \otimes L_{< p^{k(J)}}$ (compare \eqref{cool}) and hence an induced map $J_1 \otimes \cdots \otimes J_n \to J_1 \otimes \cdots \otimes J_n \otimes L_{< p^{k(J)}}$ equal to the identity on factors $J_i = L_{< p^k}$ for some $k$ and  $i = 1, \ldots, n$. 
Using the proof of Proposition \ref{Vandermonde}  the images of these maps for different $\phi$ span $J_1 \otimes \cdots \otimes J_n \otimes L_{< p^{k(J)}}$. 
Since $J_1 \otimes \cdots \otimes J_n \subset \Psi^n (\mathscr{L}_*^{n})$ by induction we conclude $J_1 \otimes \cdots \otimes J_n \otimes L_{< p^{k(J)}} \subset  \Psi^{n+1} (\mathscr{L}_*^{n+1})$ by the functoriality of group homology.
\end{proof} 

\begin{rem} The formulation of Proposition \ref{description} is inspired by \cite{JW}*{Theorem 5.1}, also see \cite{Ha15}*{Theorem 1.2}.  
In contrast to these sources the  algebraic argument above does not rely on the solution of a Conner-Floyd conjecture for $\Omega^{\SO}_*(B\Gamma)$, which seems to be inaccessible at present for $\alpha > 1$. 
Indeed we believe that our approach may be a first step towards an algebraic proof of the Conner-Floyd conjecture (for $\alpha=1$), which was resolved in \cites{Mitchell, RW} by topological methods. 
\end{rem}

\begin{proof}[Proof of Proposition \ref{productlens}] 
The decomposition of $(C(\alpha)_*)^n$ from  \eqref{dirsum} (after tensoring with $\Z/p^\alpha$, respectively $\F_p$) is compatible with the operations $\beta^{(\alpha)}$ and $\partial^{(\kappa, \alpha)}$. 
By induction on $n$ it is hence sufficient to show that the image of the composition 
\[
   \psi :   \RH_*(B\Gamma ; \Z/p^{\alpha} ) \subset (C(\alpha)_* \otimes \Z/p^\alpha)^n  \to (\tilde C(\alpha)_* \otimes \F_p)^n = (C_*)^n 
\]
is generated by generalized products of $\Z/p^{\alpha}$-lens spaces $L^{2m_1-1} \times \cdots \times L^{2m_k-1} \to B\Gamma$ for $1 \leq k \leq n$. 
Propositions \ref{impliso}  and  \ref{description} imply that the image of $\psi$ is equal to $ \mathscr{C}^{n, \infty}_* = \mathscr{L}^{n}_*$, from which this claim follows. 
\end{proof}

%%%%
\section{Proof of Theorem \ref{representable} } \label{CF} 
%%%%
Let $1 \leq \alpha_1 \leq \cdots \leq \alpha_n$, let $\Gamma = G_{\alpha_1} \times \cdots \times G_{\alpha_n}$ and let $h \in \HH_*( B \Gamma ; \Z)$ be contained in the image of  $ \Omega^{\SO}_*(B \Gamma) \to \HH_*(B\Gamma; \Z)$. 
Using the decomposition from  \eqref{dirsum} we represent $h$ by a cycle in 
\[
    C(\alpha_1)_* \otimes \cdots \otimes C(\alpha_n)_* =  \bigoplus     \tilde C(\alpha_{i_1})_*  \otimes \cdots \otimes \tilde C(\alpha_{i_k})_* \, .  
\]

\begin{prop} \label{cl} Let $1 \leq k \leq n$ and $1 \leq i_1 <  \ldots <  i_k \leq n$.
Then the  $(i_1, \ldots, i_k)$-component of this cycle in  $\tilde C(\alpha_{i_1})_*  \otimes \cdots \otimes \tilde C(\alpha_{i_k})_* \subset C(\alpha_{i_1})_* \otimes \cdots \otimes C(\alpha_{i_k})_*$ is positive.
\end{prop} 

Mapping the resulting positive classes in $\HH_*(BG_{\alpha_{i_1}} \times \cdots \times BG_{\alpha_{i_k}} ; \Z)$ to $H_*(B\Gamma; \Z)$  by the canonical subgroup inclusions this implies that $h \in \HH_*(B \Gamma ; \Z)$ is positive, finishing the proof of Theorem \ref{representable}. 

\medskip 

\begin{proof}[Proof of Proposition \ref{cl}] It is enough to deal with the case $k = n$, the other components of $h$ are treated in an analogous fashion. 
For this aim let  $c \in \tilde C(\alpha_{1})_*  \otimes \cdots \otimes \tilde C(\alpha_{n})_*$ represent the corresponding component of $h$.

Let  $1 \leq n' \leq n$ be maximal with $\alpha_{n'} = \alpha_1$, that is $\alpha_1 = \cdots = \alpha_{n'} < \alpha_{n'+1}  \leq \cdots \leq \alpha_n$. 
We set $\Gamma' := (G_{\alpha_1})^{n'}$, regarded as a subgroup of $\Gamma$ in the obvious way. 
By Proposition \ref{generated} we may assume that  $c$  is a linear combination of special cycles 
\[
    \mathscr{T}  (c_{2m_{1}-1}^{(i_1 = 1)} ,  \ldots, c_{2m_{j}-1}^{(i_j)}) \otimes c_{2d_{1}-1}^{(s_1)} \otimes \cdots \otimes c_{2d_{n-j}-1}^{(s_{n-j})} \, . 
\]
We write 
\begin{equation} \label{decomposec} 
    c = c' \otimes c^{(n'+1)}_1 \otimes \cdots \otimes c^{(n)}_1 + \mathscr{R} 
\end{equation} 
where $c'$ is a linear combination of special cycles in $ (\tilde C(\alpha_1)_*)^{n'}$ and $\mathscr{R}$ is a linear combination of special cycles 
\[
  \mathscr{T}  (c_{2m_{1} - 1}^{(i_1 = 1)} ,  \ldots, c_{2m_{j}-1}^{(i_j)}) \otimes c_{2d_{1}-1}^{(s_1)} \otimes \cdots \otimes c_{2d_{n-j}-1}^{(s_{n-j})} \in \tilde C(\alpha_{1})_*  \otimes \cdots \otimes \tilde C(\alpha_{n})_*
\]
such that $i_j \geq  n'+1$ or there exists $1 \leq \mu \leq n-j$ with $s_{\mu} \geq n' +1$ and  $d_{\mu} \geq 2$.

By Proposition \ref{bplreppos}  \ref{um} and Corollary \ref{summary} \ref{oans} the cycle $\mathscr{R} \in \tilde C(\alpha_1)_* \otimes \cdots \otimes \tilde C(\alpha_n)_* \subset C(\alpha_1)_* \otimes \cdots \otimes C(\alpha_n)_*$ is positive. 
For completing the proof of the positivity of $c$ it hence remains to show that also the cycle $c' \in (\tilde C(\alpha_1)_*)^{n'} \subset (C(\alpha_1)_*)^{n'}$ is positive. 

We will argue that by the results of Section \ref{realis} the cycle $c'$ is positive modulo some $p$-divisible, $p$-atoral cycle, which can then be dealt with by Proposition \ref{bplreppos} \ref{dois}.  
The next lemma ensures the crucial property of $c'$ needed for this argument. 
 
 \begin{lem} \label{divide} We have $[c'] \in \RH_* (B \Gamma'; \Z/p^{\alpha_1})$.
\end{lem} 

\begin{proof} Since $[c']$ lifts to an integral class it lies in the kernel of the Bockstein operation $\beta^{(\alpha_1)} : \HH_*(B \Gamma' ; \Z/p^{\alpha_1}) \to \HH_*(B \Gamma' ; \F_p)$.
In the remainder of this proof we will work with coefficients $\Z/p^{\alpha_1}$. 
It remains to show that $\partial^{(\kappa, \alpha_1)}(c') = 0 \in (C(\alpha_1)_*)^{n'}$ for all $\kappa \geq 1$. 

Since by Proposition \ref{compute_diff} the cycle $c_0$ is not hit by a differential $\partial^{(\kappa, \alpha_1)}$ it is sufficient to show this vanishing property after projection to $(\tilde C(\alpha_1)_*)^{n'}$. 
The class $[c] \in  \HH_* ( \tilde C(\alpha_{1})_*  \otimes \cdots \otimes \tilde C(\alpha_{n})_*) \cong \tilde \HH_*(\widehat{B\Gamma})$ is equal to the image of $h$ under  the projection $B\Gamma \to \widehat{B \Gamma}$. 
Since $h$ lifts to $\Omega_*^{\SO}(B \Gamma)$ we conclude that $\partial^{(\kappa, \alpha_1)}(c) = 0  \in \tilde C(\alpha_1)_* \otimes \cdots \otimes \tilde C(\alpha_n)_*$.
(Recall that $\alpha_{n'+1}, \ldots, \alpha_n > \alpha_1$ and we use coefficients $\Z/p^{\alpha_1}$.)
Since $\partial^{(\kappa, \alpha_1)}$ acts as a derivation and trivially on $c_1^{(s)}$ for $n'+1 \leq s \leq n$,   the claim $\partial^{(\kappa, \alpha_1)}(c') = 0 \in (\tilde C(\alpha_1)_*)^{n'}$ therefore follows from the assertion
\[
     \partial^{(\kappa, \alpha_1)}(\mathscr{R}) \in  \mathscr{S}_*  := {\rm span} \{  c_{d_1} \otimes \cdots \otimes c_{d_n} \mid  d_i > 1 \text{ for some }  n'+1 \leq i \leq n \}  \, . 
\]
In order to show this assertion let $\tau$ be one of the special cycles appearing in $\mathscr{R}$.
The following computations are  based on Proposition \ref{compute_diff} with $\ell  = \alpha_1$ and $\alpha \in \{ \alpha_1, \ldots, \alpha_n\}$. 

If there exists a $1 \leq \mu \leq n-j$ with $s_{\mu} \geq n'+1$ and $d_{\mu} \geq 2$, then $ \partial^{(\kappa, \alpha_1)}( \tau) \in \mathscr{S}_*$, as $\partial^{(\kappa, \alpha_1)}$ acts as a derivation and $\partial^{(\kappa, \alpha_1)} ( c^{(s_\mu)}_{2d_{\mu}-1}) = 0$.  

We will now consider the case $i_j \geq n'+1$. 
By definition 
\[
  \mathscr{T} (c_{2m_{1}-1}^{(i_1=1)} , \ldots, c_{2m_{j}-1}^{(i_j)})   = \sum_{\gamma=1}^j p^{\alpha_{i_\gamma} - \alpha_{i_1}} \cdot ( c^{(i_1=1)}_{2m_{1} } \otimes \cdots \otimes c^{(i_\gamma)}_{2m_{\gamma} -1} \otimes \cdots \otimes c^{(i_j)}_{2m_{j}} ) \, . 
\]
We distinguish the following cases: 
  \begin{itemize} 
     \item Let $1 \leq i_\gamma \leq n'$. 
              Since $\alpha_{i_j} > \alpha_1$, hence $\partial^{(\kappa, \alpha_1)}(c^{(i_j)}_{2m_{j}}) = 0$, we see  that 
   \[
       \partial^{(\kappa, \alpha_1)} ( c^{(i_1)}_{2m_1} \otimes \cdots \otimes c^{(i_\gamma)}_{2m_{\gamma} -1} \otimes \cdots \otimes c^{(i_j)}_{2m_{j}} )
   \]
  is a sum of elementary tensors each of which contains a component $c^{(i_j)}_{2m_{j}}$.   
                   \item Let $n'+1  \leq  i_\gamma \leq n$. 
              Then $\alpha_{i_\gamma} - \alpha_{i_1} \geq 1$ and 
   \[
      \partial^{(\kappa, \alpha_1)} ( p^{\alpha_{i_\gamma} - \alpha_{i_1}} \cdot c^{(1)}_{2m_1} \otimes  \cdots \otimes c^{(i_\gamma)}_{2m_{\gamma} -1} \otimes \cdots \otimes c^{(i_j)}_{2m_{j}} ) = 0.
   \] 
    \end{itemize} 
We conclude $ \partial^{(\kappa, \alpha_1)}(\tau) \in \mathscr{S}_*$ in the case $i_j \geq n'+1$ as well, finishing the proof of Lemma \ref{divide}.  
\end{proof} 

By Lemma \ref{divide} and Proposition \ref{productlens}  the image of $[c']$ in $\HH_* ( B \Gamma'; \F_p)$ is a linear combination of generalized products of $\Z/p^{\alpha_1}$-lens spaces $L^{2m_1-1} \times \cdots \times L^{2m_k-1} \to B\Gamma'$ for $1 \leq k \leq n'$. 

The cycle $\mathscr{R}$ occurring in  \eqref{decomposec} is $p$-atoral, since each summand is of degree larger than $n$ and $p$ is odd. 
Since also $c$ is $p$-atoral we conclude that $c' \in (C(\alpha_1)_*)^{n'}$ is $p$-atoral (here we use again that  $\alpha_{n'+1}, \ldots,  \alpha_n >  \alpha_1$). 
We can therefore assume that in the generalized products of lens spaces appearing before the case $m_1 = \ldots = m_k =1$ does not occur. 

In summary, modulo some $p$-atoral positive cycle (represented by a linear combination of generalized products of lens spaces) we can assume that  $c'$ is a $p$-atoral cycle in $(\tilde C(\alpha_1)_*)^{n'} \subset (C(\alpha_1)_*)^{n'}$ that maps to $0 \in \HH_* (B \Gamma' ;  \F_p)$. 

Since $\HH_* (B \Gamma'; \Z) \otimes \F_p \to \HH_* ( B \Gamma'; \F_p ) $ is injective we get $[c'] = p \cdot [\xi]$ for a cycle $\xi  \in  (\tilde C(\alpha_1)_*)^{n'}$. 
According to  Proposition  \ref{generated} we can assume that $\xi$ is a linear combination of special cycles, and hence $p \cdot \xi$ is a linear combination of cycles 
\[
  p \cdot  \big( \mathscr{T}  (c_{2m_{1}-1}^{(i_1)} , \ldots, c_{2m_{j}-1}^{(i_j)}) \otimes c_{2d_{1}-1}^{(s_1)} \otimes \cdots \otimes c_{2d_{n'-j}-1}^{(s_{n'-j})} \big) \in (\tilde C(\alpha_1)_*)^{n'} \, . 
\]
Proposition  \ref{bplreppos} \ref{dois} implies  that such cycles are positive in $(C(\alpha_1)_*)^{n'}$ whenever $j \geq 2$.
Since these cycles are also $p$-atoral for $j \geq 2$ (for $p$-odd) we can assume that $p \cdot \xi$ is a linear combination of cycles $c_{2d_1-1} \otimes \cdots \otimes c_{2d_{n'}-1}$ with at least one $d_i \geq 2$ (by $p$-atorality of $c'$). 
This shows that $p \cdot \xi$, and hence $c'$ are positive.  

In summary we have shown that $c \in \tilde C(\alpha_1)_* \otimes \cdots \otimes \tilde C(\alpha_n)_* \subset C(\alpha_1)_* \otimes \cdots \otimes C(\alpha_n)_*$ is positive, finishing the proof of Proposition \ref{cl} and hence of Theorem \ref{representable}.    

\end{proof}

\end{document}